\documentclass[12pt]{article}
\usepackage{}
\usepackage{amsfonts}
\usepackage{amssymb}
\usepackage{mathrsfs}
\usepackage{srcltx}
\usepackage{amsmath}
\usepackage{amsthm}
\usepackage{amstext}
\usepackage{amsopn}
\usepackage{graphicx}
\usepackage{bm}
\newtheorem{theorem}{Theorem}[section]
\newtheorem{lemma}[theorem]{Lemma}
\newtheorem{proposition}[theorem]{Proposition}

\theoremstyle{definition}

\theoremstyle{remark}

\numberwithin{equation}{section}

\newcommand{\ba}{\begin{array}}
\newcommand{\ea}{\end{array}}
\newcommand{\f}{\frac}

\newcommand{\ds}{\displaystyle}

\begin{document}
\date{}
\title{ \bf\large{The stability and Hopf bifurcation of the diffusive Nicholson's blowflies model in spatially heterogeneous environment}\thanks{This research is supported by the National Natural Science Foundation of China (No 11771109).}}
\author{Dan Huang\textsuperscript{1},\ \ Shanshan Chen\textsuperscript{2}\footnote{Corresponding Author, Email: chenss@hit.edu.cn}\ \
 \\
{\small \textsuperscript{1} School of Mathematics, Harbin Institute of Technology,\hfill{\ }}\\
\ \ {\small  Harbin, Heilongjiang, 150001, P.R.China.\hfill{\ }}\\
{\small \textsuperscript{2} Department of Mathematics, Harbin Institute of Technology,\hfill{\ }}\\
\ \ {\small Weihai, Shandong, 264209, P.R.China.\hfill{\ }}\\
}
\maketitle

\begin{abstract}
In this paper, we consider the diffusive Nicholson's blowflies model in spatially heterogeneous environment when the diffusion rate is large. We show that the ratio of the average
of the maximum per capita egg production rate to that of the death rate affects the dynamics of the model. The unique positive steady state is locally asymptotically stable if the ratio is less than a critical value. However, when the ratio is greater than the critical value, large time delay can make the unique positive steady state unstable through Hopf bifurcation. Especially, the first Hopf bifurcation value tends to that of the \lq\lq average\rq\rq~DDE model when the diffusion rate tends to infinity. Moreover, we show that the direction of the Hopf bifurcation is forward, and the bifurcating periodic
solution from the first Hopf bifurcation value is orbitally asymptotically stable, which improves the earlier result by Wei and Li (Nonlinear. Anal., 60: 1351--1367, 2005).

\end{abstract}

\section{Introduction}

To explain the oscillatory behavior of blowfly observed by Nicholson, Gurney et al. \cite{Gurney1980Nicholson} proposed the following classical Nicholson's blowflies model,
\begin{equation}\label{ddem}
\frac{{d u(t)}}{{d t}} = pu(t - \hat \tau){e^{ - au(t - \hat \tau )}} - \delta u(t),
\end{equation}
where $u(t)$ represents the size of the adult population in time $t$, $p$ is the maximum per capita egg production rate, $1/a$ is the size at which the population reproduces at its maximum rate, $\delta$ is the per capita daily death rate, and time delay $\hat\tau$ represents the maturation (or generation) time. The global dynamics of model \eqref{ddem} has been investigated extensively, see \cite{FengYan,Trofimchuk,HouDuanHuang,Liu2014,ShiSong,ShuLiWu,SoYu,WeiLi} and references therein.

Considering the spatial environment, one could obtained the following diffusive Nicholson's blowflies model:
\begin{equation}\label{difNi}
\begin{cases}
\ds\frac{\partial u}{\partial t} = d\Delta u + pu(x,t -\hat \tau ){e^{ - au(x,t - \hat\tau)}} - \delta u(x,t),& x \in \Omega,\;t > 0,\\
\displaystyle  {\partial _n}u=0,&x\in\partial\Omega,\;t>0,
\end{cases}
\end{equation}
where $\Omega$ is a bounded domain with a smooth boundary $\partial \Omega$, $n$ is the outward unit normal vector on $\partial \Omega$, and $\Delta$ is the Laplacian operator which models the passive movement of the species in space. It was proved in \cite{YangSo1996} that all non-trivial solutions
of \eqref{difNi} converge to the unique positive steady state for $1<p/\delta<e$, and
when $p/\delta>e^2$, the large time delay $\hat\tau$ could make the unique positive steady state unstable through Hopf bifurcation.
Yi and Zou \cite{YiZou2008} showed that the unique positive steady state of model \eqref{difNi} is also globally attractive for the non-monotone case that $e<p/\delta\le e^2$.
Gourley and Ruan \cite{GourleyRuan2000} considered the global dynamics of model \eqref{difNi} with the distributed delay.
Model \eqref{difNi} with the homogeneous Dirichlet boundary condition was also studied extensively, see \cite{GuoMa2016,SoYang1998,SuShiWei2010} for the global dynamics and
Hopf bifurcation. We point out that the results on the travelling wave solution were obtained in \cite{LiRuanWang2007,LinLin2014,MeiSo2004,SoZou2001,ZhangPeng2008} and references therein when $\Omega$ is unbounded.

In model \eqref{difNi}, all the parameters are constant. Due to the heterogeneity of the environment, the blowflies may have different egg production rates or death rates at difference spaces.
Then we consider the following spatially heterogeneous model:
\begin{equation}\label{M1}
\left\{ \begin{array}{ll}
\displaystyle\frac{\partial u}{\partial t} = d\Delta u+ p(x)u(x,t - \hat\tau ){e^{ - au(x,t -\hat \tau )}} - \delta(x) u(x,t),& x \in \Omega,\;\;t > 0,\\
\displaystyle  {\partial _n}u = 0,& x \in \partial \Omega,\;\;t > 0.
\end{array} \right.
\end{equation}
To avoid unnecessary complications, here we only assume that the maximum per capita egg production rate $p$ and
the per capita death rate $\delta$ are positive and spatially dependent, and all the other parameters are positive constants.
If $d=0$, model \eqref{M1} could be regarded as a system of DDE, and clearly for every $x\in\Omega$, the positive equilibrium $\ds\frac{1}{a}\ln \ds\frac{p(x)}{
\delta(x)}$ is globally attractive if
$1<p(x)/\delta(x)\le e^2$, and
large delay $\hat\tau$ could induce Hopf bifurcation if $p(x)/\delta(x)>e^2$.
A natural question is whether large delay could induce Hopf bifurcation for
 model \eqref{M1} when $d\ne0$, and in this paper, we will consider this problem when the diffusion rate $d$ is large.

The main method in this paper are motivated by \cite{Busenberg1996Stability}, where Busenberg and Huang showed the existence of the Hopf bifurcation near the spatially nonhomogeneous steady state. Since then, there exist extensive results on the Hopf bifurcation near such type of spatially nonhomogeneous
steady state, see \cite{Chen2012Stability,Chen2016Stability,Guo2015,Guo2017,GuoYan2016,HuYuan2011,SuWeiShi2009,SuWeiShi2012,YanLi2010} and references therein.
Moreover, this method could also be applied to the delayed logistic population model in spatially heterogeneous environment \cite{Chen2018Hopf,ChenWeiZhang,ShiShiSong2019}, and large time could induce Hopf bifurcation. For the Nicholson's blowflies model, delay cannot always induce Hopf bifurcation, and we need to modify some arguments in \cite{Busenberg1996Stability,ChenWeiZhang}.

Letting $\tilde u = au, \tilde t = d t $ and $f(\tilde u) = \tilde u{e^{ -\tilde u}}$, denoting $ r =1/d$, $\tau  = d \hat\tau$, and dropping the tilde sign, system \eqref{M1} can be transformed as follows:
\begin{equation}\label{M2}
\left\{ \begin{array}{ll}
\displaystyle\frac{{\partial u}}{{\partial t}} = \Delta u+ r p(x) f(u(x,t - \tau )) - r \delta(x) u,& x \in \Omega ,\;\;t > 0,\\
\displaystyle  {\partial _n}u = 0,& x \in \partial \Omega ,\;\;t > 0.
\end{array} \right.
\end{equation}
It follows from \cite[Theorem 2.5, and Propositions 3.2 and 3.3]{CantrellCosner} that
\begin{proposition}\label{pp1}
Assume that $c_0>0$ and $r>0$, where
\begin{equation}\label{c0}
c_0=\ln \frac{\bar p}{\bar \delta },\;\;\bar \delta  =\ds\frac{\int_{\Omega}\delta(x)dx}{|\Omega|}\;\;\text{and}\;\;\bar p  =\ds\frac{\int_{\Omega}p(x)dx}{|\Omega|}.
\end{equation}
Then model \eqref{M2} admits a unique positive steady state $u_r$, which is globally asymptotically stable for $\tau=0$.
\end{proposition}
By the similar arguments as in \cite{Cantrell1996,ChenWeiZhang}, we have the asymptotic profile of $u_r$.
\begin{proposition}\label{pp2}
Assume that $c_0>0$. Then $\mathop {\lim {u_r}}\limits_{r \to 0}  = {c_0}$, and $u_r$ is continuously differentiable for $r\in[0,\infty)$ if  $u_0\equiv c_0$.
\end{proposition}
Our main results are summarized as follows (see Theorems \ref{distribution of eigenvalues} and \ref{stabliHopf}): for $r \in(0, r_1]$, where $0<r_1 \ll 1$,
 \begin{enumerate}
 \item [$(i)$] if $0<c_0<2$, then the unique positive steady state $u_r$ of model \eqref{M2} is locally asymptotically stable for any $\tau\ge0$;
 \item [$(ii)$] if $c_0>2$,  there exists a sequence $\{\tau_n\}_{n=0}^\infty$ such that
$u_r$ is locally asymptotically stable for any $\tau\in[0,\tau_0)$ and unstable for $\tau>\tau_0$, and model \eqref{M2} occurs Hopf bifurcation at $u_r$ when $\tau=\tau_n$ ($n=0,1,\dots$). Moveover, for each $n \in \mathbb{N} \cup\{0\}$, the direction of the Hopf bifurcation at $\tau=\tau_{n}$ is forward, that is, the bifurcating periodic solutions exist for $ \tau>\tau_n$, and the bifurcating periodic solution from $\tau=\tau_{0}$ is orbitally asymptotically stable.
\end{enumerate}
Therefore, we have the associated results for the equivalent model \eqref{M1}, when diffusion rate $d$ is large, see Proposition \ref{equivthm37}.

Throughout the paper, as in \cite{ChenWeiZhang}, we denote the spaces $X=\{u\in H^2(\Omega) :{\partial _n}u = 0\}, Y=L^2(\Omega), C=C([-\tau,0], Y)$, and $\mathcal {C}=C([-1,0], Y)$. It is well known that
\begin{equation}\label{X1}
X=\mathscr{N}(\Delta)\oplus X_1,\; Y=\mathscr{N}(\Delta)\oplus Y_1,
\end{equation}
where
\begin{equation}\label{X1Y1}
\begin{split}
\displaystyle \mathscr{N}(\Delta) =& span\{ \phi \}  = span\{ 1\} ,\; X_1 = \{ y \in X:\int_0^{ L} y(x) dx = 0\}, \\
\displaystyle {Y_1} =&\mathscr{R}(\Delta) = \{ y \in Y:\int_0^{ L} {y(x)} dx = 0\}.
\end{split}
\end{equation}
For any subspace $Z$ of $X, Y , C$ or $\mathcal {C}$, let the complexification of $Z$ be $Z_\mathbb{C}:= Z \oplus {\rm i}Z = \{x_1+{\rm i}x_2 | x_1, x_2 \in Z\}$. Define the domain of a linear operator $T$ by $\mathscr{D}(T)$, the kernel of $T$ by $\mathscr{N}(T)$, and the range of $T$ by $\mathscr{R}(T)$.
For the Hilbert space $Y_{\mathbb{C}}$, we choose the standard inner product $\langle u,v\rangle =\int\limits_\Omega  {\bar u(x)v(x)} dx$.

The rest of the paper is organized as follows. In Section 2, we study the stability/instability of the unique positive steady state of Eq. \eqref{M2} and the associated Hopf bifurcation. In Section 3, we analyze the direction of the Hopf bifurcation and the stability of and the bifurcating periodic
solutions. In Section 4, we give the associated results for the equivalent model \eqref{M1}, and some numerical simulations are given to illustrate our theoretical results.

\section{Stability and Hopf bifurcation}

In this section, we consider the stability of the positive steady state $u_r$ and the associated Hopf bifurcation.
Linearizing model \eqref{M2} at $u_r$, we have
\begin{equation}\label{Linearized system}
\begin{cases}
\displaystyle \frac{\partial v}{\partial t}=\Delta v+r p(x) f^{\prime}\left(u_{r}\right) v(x, t-\tau)- r \delta(x) v, & x \in \partial \Omega,\;\; t>0, \\
\partial _nv = 0, & x \in \partial \Omega,\;\; t>0.
\end{cases}
\end{equation}
It follows from \cite[Chapter 3]{Wu1996Theory} that the infinitesimal generator $A_{\tau} (r)$ of the solution semigroup of Eq. \eqref{Linearized system} satisfies
\begin{equation}\label{Ataur}
A_{\tau}(r) \Psi=\dot{\Psi},
\end{equation}
where
\begin{equation*}
\begin{split}
 &\mathscr{D}\left(A_{\tau}(r)\right)=\big\{\Psi \in C_{\mathbb{C}} \cap C_{\mathbb{C}}^{1}: \Psi(0) \in X_{\mathbb{C}}, \dot{\Psi}(0)=\Delta \Psi(0)\\
 &~~~~~~~~~~~~~~~~~~+r p(x) f'\left(u_{r}\right)\Psi(-\tau)- r \delta(x) \Psi(0)\big\},
\end{split}
\end{equation*}
and $C_{\mathbb{C}}^{1}=C^{1}\left([-\tau, 0], Y_{\mathbb{C}}\right)$. Therefore, $\mu \in \mathbb{C}$ is an eigenvalue of $A_{\tau}(r),$ if and only if there exists $\psi(\neq 0) \in X_{\mathbb{C}}$ such that $\Delta(r, \mu, \tau) \psi=0,$ where
\begin{equation}\label{Deltamu}
\Delta(r, \mu, \tau) \psi:=\Delta \psi+r e^{-\mu \tau} p(x) f^{\prime}\left(u_{r}\right) \psi- r \delta(x) \psi-\mu \psi.
\end{equation}

Firstly, we give the following estimates for solutions of Eq. \eqref{Deltamu}.
\begin{lemma}\label{bounded}
Assume that $\left(\mu_{r}, \tau_{r}, \psi_{r}\right)$ solves Eq. \eqref{Deltamu} with $\mathcal{R}e \mu_{r}, \tau_{r} \ge 0$ and $\psi_r(\neq 0) \in X_{\mathbb{C}}$, then $\left |\displaystyle\frac{\mu_{r}}{r} \right |$ is bounded for $r \in (0,r_1]$.
\end{lemma}

\begin{proof}
Multiplying $\Delta(r, \mu_r, \tau_r) \psi_r=0$ by $ \overline {\psi}_r,$ and integrating the result over $\Omega$, yields
$$\left \langle \psi_r,\Delta \psi_r\right\rangle + r \int_{\Omega} p(x) f^{\prime}\left(u_{r}\right) |\psi_r|^2 dx e^{-\mu_{r} \tau_{r}} -r  \int_{\Omega}\delta(x)|\psi_r|^2 dx -\mu_{r} \int_{\Omega}|\psi_r|^2 dx =0.$$
Without loss of generality, we assume that $\|\psi_r\|_{Y_\mathbb{C}}^{2} =1$. Noticing that
$$
\left\langle\psi_r, \Delta \psi_r\right\rangle=-\int_{\Omega} \left|\nabla \psi_r\right|^{2} d x \le 0,
$$
we obtain that
$$\left \langle \psi_r,\Delta \psi_r\right\rangle =- r \int_{\Omega} p(x) f^{\prime}\left(u_{r}\right) |\psi_r|^2 dx e^{-\mu_{r} \tau_{r}} +r  \int_{\Omega}\delta(x)|\psi_r|^2 dx +\mu_{r} \le 0.$$
Therefore,
$$
\begin{aligned}
0  \leq  \mathcal{R}e\left(\frac{\mu_r}{r}\right) & \leq \mathcal{R}e \left[ \int_{\Omega} p(x) f^{\prime}\left(u_{r}\right) |\psi_r|^2 dx e^{-\mu_{r} \tau_{r}} - \int_{\Omega}\delta(x)|\psi_r|^2 dx  \right] \\
& \leq \max _{\Omega} p(x) \left\|f^{\prime}\left(u_{r}\right) \right\|_{\infty}. 
\end{aligned}
$$
Similarly, for imaginary parts, we have
$$
\left|\mathcal{I} m\left(\frac{\mu_r}{r}\right)\right|\leq \max _{\Omega} p(x) \left\|f^{\prime}\left(u_{r}\right) \right\|_{\infty}.
$$
It follows from the continuity of mapping $r \mapsto u_{r}$ that $\left |\displaystyle\frac{\mu_{r}}{r} \right |$ is bounded for $r \in (0,r_1]$.
\end{proof}
The following result is similar to \cite[Lemma 2.3]{Busenberg1996Stability}, and we omit the proof.
\begin{lemma}\label{r2}
If $z \in X_{\mathbb{C}}$ and $\langle\phi, z\rangle= 0,$ then $|\langle \Delta z, z\rangle| \geq r_{2}\|z\|_{Y_\mathbb{C}}^{2},$ where $r_{2}$ is the second eigenvalue of operator $-\Delta$.

\end{lemma}

By virtue of the similar arguments as in \cite[Theorem 3.3]{Chen2016Stability}, we see from Lemmas \ref{bounded} and \ref{r2} that:
\begin{theorem}\label{sstability}
Assume that $0<c_0<2$. Then there exists $r_1>1$ such that $\sigma\left(A_{\tau}(r)\right) \subset\{x+ {\rm i} y: x, y \in \mathbb{R}, x<0\}$ for $ r \in (0,r_1].$
\end{theorem}
\begin{proof}
By way of contradiction, there exists a positive sequence $\left\{r_{n}\right\}_{n=1}^{\infty}$ such that $\lim _{n \rightarrow \infty} r_{n}=0,$ and, for $n \geq 1, \Delta\left(r_{n}, \mu, \tau\right) \psi=0$ is solvable for some value of
$\left(\mu_{{n}}, \tau_{{n}}, \psi_{{n}}\right)$ with $\mathcal{R} e \mu_{{n}}, \mathcal{I} m \mu_{{n}} \geq 0, \tau_{{n}} \geq 0$ and $0 \neq \psi_{{n}} \in X_{\mathbb{C}}$. Ignoring a scalar factor, we see from Eq. \eqref{X1} that $\psi_{{n}}$ can be represented as
\begin{equation}\label{psirn}
\begin{split}
&\psi_{{n}}=\beta_{{n}} c_{0}+r_n z_{{n}},  \; z_{{n}} \in\left(X_{1}\right)_{\mathbb{C}}, \; \beta_{{n}} \geq 0, \\
&\|\psi_{{n}}\|_{Y_{\mathbb{C}}}^{2}=\beta_{{n}}^{2} c_{0}^{2}|\Omega|+r_n^{2}\|z_{{n}}\|_{Y_{\mathbb{C}}}^{2}=c_{0}^{2}|\Omega|.
\end{split}
\end{equation}
Substituting \eqref{psirn} and $\mu_{{n}}=r_n h_{{n}}$ into $\Delta(r_n, \mu_{r_{n}}, \tau_{{n}}) \psi_{{n}}=0$ , we obtain that
\begin{equation}
\begin{split}
&H_{1}(z_n, \beta_n, h_n, \tau_n, r_n):=\Delta z_n+\left(e^{-r_n h_n \tau_n} p(x) f'\left(u_{r_n}\right)-\delta(x)-h_n\right)(\beta c_{0}+r_n z_n)=0, \\
&H_{2}(z_n, \beta_n, r_n):=\left(\beta_n^{2}-1\right) c_{0}^{2}|\Omega|+r_n^{2}\|z_n\|_{Y_{\mathbb{C}}}^{2}=0,
\end{split}
\end{equation}
It follows from Lemma \ref{bounded} and Eq. \eqref{psirn} that $ |h_{{n}}|=\left |\ds\frac{\mu_{{n}}}{r_n} \right|$ is bounded and $|\beta_{{n}}| \le 1 $. It follows from Lemma \ref{r2} that there are $M_1, M_2 > 0$ such that
$$
r_{2}\left\|z_{{n}} \right\|_{Y_{\mathbb{C}}}^{2} \leq|\langle \Delta z_{{n}} , z_{{n}} \rangle| \leq M_{1}\left\|z_{{n}} \right\|_{Y_{\mathbb{C}}}+M_{2}r\left\|z_{{n}} \right\|_{Y_{\mathbb{C}}}^{2},
$$
where $r_{2}$ is defined as in Lemma \ref{r2}. Therefore, for sufficiently small $\overline r$, $\left\{z_{{n}}\right\}_{n=1}^{\infty}$ is bounded in $Y_{\mathbb{C}}$ for $r \in\left[0, \overline r\right] .$ Since the operator $\Delta: (X_1)_{\mathbb{C}} \mapsto  (Y_1)_{\mathbb{C}}$ has a bounded inverse, and by applying $\Delta ^{-1}$ on $H_1(z_{{n}}, \beta_{{n}}, h_{{n}}, \tau_{{n}} ,r_n)=0$, we obtain that $\left\{z_{{n}}\right\}_{n=1}^{\infty}$ is also bounded in $\left(X_{1}\right)_{\mathbb{C}}$. Therefore, we see that
$$\left\{\left(z_{{n}}, \beta_{{n}}, h_{{n}}, e^{- r_{n} \tau_{{n}}\left(\mathcal{R}e h_{{n}}\right)}, e^{-{\rm i} r_{n} \tau_{{n}}\left(\mathcal{I}m h_{{n}}\right)}\right)\right\}_{n=1}^{\infty}$$
is precompact in $Y_{\mathbb{C}} \times \mathbb{R}^{3} \times \mathbb{C}.$ Then, there exists a subsequence

$$\left\{\left(z_{{{n}_{k}}}, \beta_{_{{n}_{k}}}, h_{{{n}_{k}}}, e^{- r_{{n}_{k}} \tau_{{{n}_{k}}}\left(\mathcal{R}e h_{{{n}_{k}}}\right)}, e^{-{\rm i} r_{{n}_{k}} \tau_{{{n}_{k}}}\left(\mathcal{I}m h_{{{n}_{k}}}\right)}\right)\right\}_{k=1}^{\infty}$$
convergence to $(z^*,\beta^*,h^*,\sigma^*,e^{-{\rm i}\theta^* })$ as $k \rightarrow \infty $ in the norm of $Y_{\mathbb{C}} \times \mathbb{R}^{3} \times \mathbb{C},$ where
$$
\beta^{*}=1, \;\;z^{*} \in Y_{\mathbb{C}},\;\; h^{*} \in \mathbb{C}\left(\mathcal{R} e h^{*},\mathcal{I} m h^{*} \geq 0\right),\; \;\theta^{*} \in[0,2 \pi) \;\;\text {and}\;\; \sigma^{*} \in[0,1].
$$
Taking the limit of the equation $\Delta ^{-1} H_1 \left(z_{{{n}_{k}}}, \beta_{{{n}_{k}}}, h_{{{n}_{k}}}, \tau_{{{n}_{k}}}, r_{{n}_{k}}\right)=0 $ as $k \rightarrow \infty$, we see that $z^* \in (X_1)_{\mathbb{C}}$ and $(z^*,\beta^*,h^*,\theta^*, \sigma^*)$ satisfies
$$\Delta z^*+\left(\sigma^* e^{-{\rm i}\theta^* } p(x) f'\left(c_0\right)-\delta(x)-h^*\right)c_0=0.$$
Then
\begin{equation}\label{Eqconsin}
\begin{cases}
\sigma^* f^{\prime}(c_0) \int_{\Omega}  p(x) d x \cos \theta^*=\int_{\Omega}\delta(x) d x +(\mathcal{R} e h^{*}) |\Omega|, \\
- \sigma^* f^{\prime}(c_0) \int_{\Omega}  p(x)d x \sin \theta^*=(\mathcal{I}m h^{*})|\Omega|.
\end{cases}
\end{equation}
It follows from the first equation of \eqref{Eqconsin} that
\begin{equation*}
[\sigma^*(1-c_0)]^2\ge1.
\end{equation*}
Since $0<c_0<2$ and $\sigma^*\in[0,1]$, we have $[\sigma^*(1-c_0)]^2<1$, which is a contradiction.

\end{proof}

Then we consider the case of $c_0>2$, and show that large delay will induce Hopf bifurcation.
To show the existence of Hopf bifurcation, we need to verify that the eigenvalues of $A_\tau(r)$ could pass through the imaginary axis as time delay $\tau$ increases. Clearly, $A_\tau(r)$ has a purely imaginary eigenvalue $\mu = {\rm i}\nu (\nu>0)$ for some $\tau \ge 0$, if and only if
\begin{equation}\label{Deltaiv}
\Delta \psi+r e^{-{\rm i} \theta} p(x) f^{\prime}\left(u_{r}\right) \psi- r \delta(x)\psi-{\rm i}\nu \psi=0
\end{equation}
is solvable for some value of $\nu >0, \theta \in [0, 2\pi),$ and $\psi(\neq 0) \in X_{\mathbb{C}}$.
Ignoring a scalar factor, we see from \eqref{X1} that if $(\nu,\theta,\psi)$ solves \eqref{Deltaiv}, then
$\psi\in X_{\mathbb C}$ can be represented as
\begin{equation}\label{psi}
\begin{split}
&\psi=\beta c_{0}+r z,  \; z \in\left(X_{1}\right)_{\mathbb{C}}, \; \beta \geq 0 \\
&\|\psi\|_{Y_{\mathbb{C}}}^{2}=\beta^{2} c_{0}^{2}|\Omega|+r^{2}\|z\|_{Y_{\mathbb{C}}}^{2}=c_{0}^{2}|\Omega|.
\end{split}
\end{equation}
Plugging \eqref{psi} and $\nu=rh$
into Eq. \eqref{Deltaiv}, we obtain that $(\nu,\theta,\psi)$ solves Eq. \eqref{Deltaiv}, where $\nu>0$, $\theta\in[0,2\pi)$ and $\psi\in X_{\mathbb{C}}$,
if and only if the following system:
\begin{equation}\label{g1g2}
\begin{cases}
g_{1}(z, \beta, h, \theta,r):=\Delta z+\left(e^{-{\rm i} \theta} p(x) f'\left(u_{r}\right)-\delta(x)-{\rm i}h\right)(\beta c_{0}+rz)=0 \\
g_{2}(z, \beta, r):=\left(\beta^{2}-1\right) c_{0}^{2}|\Omega|+r^{2}\|z\|_{Y_{\mathbb{C}}}^{2}=0
\end{cases}
\end{equation}
has a solution $(z,\beta,h,\theta)$, where $z\in (X_1)_{\mathbb{C}}$, $\beta\ge0$, $h>0$ and $\theta\in[0,2\pi)$.
Define
$G:(X_1)_{\mathbb C}\times \mathbb{R}^4\to
Y_{\mathbb C}\times \mathbb{R}$ by $G=(g_1,g_2)$.

We first consider the solution of $G(z, \beta, h, \theta, r)=0$ for $r=0$.

\begin{lemma}\label{z0beta0h0theta0}
Assume that $c_0 > 2$. Then the following equation
\begin{equation}
\left\{\begin{array}{l}
\displaystyle {G(z, \beta, h, \theta, 0)=0} \\
\displaystyle {z \in\left(X_{1}\right)_{\mathbb{C}},\; h,\; \beta \geq 0,\; \theta \in[0,2 \pi]}
\end{array}\right.
\end{equation}
has a unique solution $\left(z_0, \beta_0, h_0, \theta_0\right)$, where
\begin{equation}\label{cos0h0}
\begin{array}{c}
\displaystyle \cos \theta_0=\frac{1}{1-c_0},\; \sin \theta_0=- \frac{{\sqrt {c_0^2 - 2{c_0}} }}{{1 - {c_0}}},\\
\displaystyle \beta_0=1,\; h_0= \bar \delta \sqrt{c_0^2 - 2{c_0}},
\end{array}
\end{equation}
and $z_0 \in\left(X_{1}\right)_{\mathbb{C}}$ is the unique solution of
\begin{equation}\label{z0}
\Delta z= -c_0 f^{\prime}(c_0) p(x)e^{-{\rm i} \theta_0} +c_0\delta(x) + {\rm i} h_0 c_0.
\end{equation}
\end{lemma}
\begin{proof}
It follows from \eqref{g1g2} that $g_2(z,\beta,0)=0$ if and only if $\beta=\beta_0=1$. Note that
\begin{equation}
g_{1}(z, \beta_0, h, \theta, 0)=\Delta z +c_0 f^{\prime}(c_0) p(x)e^{-{\rm i} \theta} + c_0 \delta(x) + {\rm i} h c_0.
\end{equation}
Then
$$
\left\{\begin{array}{l}
g_{1}\left(z, \beta_0, h, \theta, 0\right)=0 \\
z \in\left(X_{1}\right)_{\mathbb{C}},\;h, r \geq 0,\; \theta \in[0,2 \pi]
\end{array}\right.
$$
is solvable if and only if

\begin{equation}\label{sed}
\left\{\begin{array}{l}
{f^{\prime}(c_0) \int_{\Omega}  p(x) d x \cos \theta=\int_{\Omega}\delta(x) d x} \\
{- f^{\prime}(c_0) \int_{\Omega}  p(x)d x \sin \theta=h |\Omega|}
\end{array}\right.
\end{equation}
is solvable for a pair $(\theta, h)$ with $h \ge 0$ and $\theta \in[0, 2\pi]$.
From \eqref{c0}, we see that \eqref{sed} has a unique solution $(\theta_0, h_0)$, which satisfies
\begin{equation}\label{thetah}
\begin{array}{c}
\displaystyle \cos \theta_0=\frac{1}{1-c_0},\; \sin \theta_0=- \frac{{\sqrt {c_0^2 - 2{c_0}} }}{{1 - {c_0}}},\; h_0= \bar\delta \sqrt{c_0^2 - 2{c_0}}
\end{array}
\end{equation}
when $c_0 > 2$. Therefore, $g_{1}(z, \beta_0, h_0, \theta_0, 0)=0$ has a unique solution $z_0$, which satisfies Eq. \eqref{z0}.

\end{proof}
Now we consider the case of $r\ne0$.
\begin{theorem}\label{uniquesolution}
Assume that $c_0>2$. Then there exist $\tilde r_1\in(0,r_1) $, and a continuously differentiable mapping $r \mapsto (z_r,\beta_r,h_r,\theta_r)$ from $[0,\tilde r_1]$ to $ (X_1)_\mathbb{C} \times \mathbb{R}^3$ such that $G(z_r, \beta_r, h_r, \theta_r, r) =0$. Moreover, for $r \in [0, \tilde r_1]$, $(z_r, \beta_r, h_r, \theta_r)$ is the unique solution of the following problem
\begin{equation}
\begin{cases}
G(z, \beta, h, \theta, r)=0 ,\\
z \in\left(X_{1}\right)_{\mathbb{C}},\; h>0,\; \beta \geq 0,\; \theta \in[0,2 \pi).
\end{cases}
\end{equation}

\end{theorem}

\begin{proof}
Let $T=(T_1,T_2):=(X_1)_{\mathbb{C}} \times \mathbb{R}^3 \rightarrow  Y_{\mathbb{C}} \times \mathbb{R} $ be the Fr\'{e}chet derivative of $G$ with respect to $(z,\beta,h,\theta) $ at  $(z_0,\beta_0,h_0,\theta_0,0) $. Thus, we have
$$
\begin{aligned}
T_{1}(\chi, \kappa, \epsilon, \vartheta)=& \Delta \chi-{\rm i} \epsilon c_0- {\rm i} \vartheta c_0 f^\prime(c_0) p(x)e^{-{\rm i} \theta_0} \\
&+\kappa c_0\left[ f^\prime(c_0) p(x) e^{-{\rm i} \theta_0}- c_0 \delta(x)- {\rm i} h_0 c_0 \right], \\
T_{2}(\kappa)=& 2 \kappa c_0^{2} |\Omega|.
\end{aligned}
$$
Then, we check that $T$ is a bijection from $(X_1)_{\mathbb{C}} \times \mathbb{R}^3$ to $Y_{\mathbb{C}} \times \mathbb{R} $, and we only need to verify that $T$ is an injective mapping. If $T_2(\kappa)=0$, then $\kappa=0$, and substituting $\kappa=0$ into $T_1$, i.e. $T_{1}(\chi, 0, \epsilon, \vartheta)=0$, we have $\epsilon=\vartheta=0 $. Therefore, $T$ is an injection. This, combined with the implicit function theorem, implies that there exist $\tilde r_1 >0$, and a continuously differentiable mapping $r\mapsto(z_r,\beta_r,h_r,\theta_r)$ from $[0,\tilde r_1]$ to $ (X_1)_\mathbb{C} \times \mathbb{R}^3$ such that $G(z_r, \beta_r, h_r, \theta_r, r) =0$. Next, we prove the uniqueness. We only need to verify that if
$z^r \in\left(X_{1}\right)_{\mathbb{C}}, \beta^r\ge0, h^r>0 ,\theta^r \in[0,2 \pi)$, and $G(z^r, \beta^r, h^r, \theta^r, r)=0 $, then
$$
\left(z^{r}, \beta^{r}, h^{r}, \theta^{r}\right) \rightarrow\left(z_{0}, \beta_{0}, h_{0}, \theta_{0}\right)=\left(z_0, 1, h_{0}, \theta_0 \right)
$$
as $r \rightarrow 0$ in the norm of $ (X)_\mathbb{C} \times \mathbb{R}^3$. It follows from Lemma \ref{bounded} and Eq. \eqref{g1g2} that $\{h^r\},\{\beta^r\}$ and $\{\theta^r\}$ are bounded for $r \in [0, \tilde r_1]$. From Lemma\ref{r2}, we can calculate that there are $M_1, M_2 > 0$ such that
$$
r_{2}\left\|z^r\right\|_{Y_{\mathbb{C}}}^{2} \leq|\langle \Delta z, z\rangle| \leq M_{1}\left\|z^{r}\right\|_{Y_{\mathbb{C}}}+M_{2}r\left\|z^r\right\|_{Y_{\mathbb{C}}}^{2},
$$
where $r_{2}$ is defined as in Lemma \ref{r2}. So, for sufficiently small $\tilde r_1$, $\left\{z^r\right\}$ is bounded in $Y_{\mathbb{C}}$ for $r \in\left[0, \tilde r_1\right] .$ Since the operator $\Delta: (X_1)_{\mathbb{C}} \mapsto  (Y_1)_{\mathbb{C}}$ has a bounded inverse, and by applying $\Delta ^{-1}$ on $g_1(z^{r}, \beta^{r}, h^{r}, \theta^{r} ,r)=0$, we obtain that $\left\{z^r\right\}$ is also bounded in $\left(X_{1}\right)_{\mathbb{C}}$. Therefore, we see that $\left\{\left(z^r, \beta^r, h^r, \theta^r\right): r \in\left(0, \tilde{r}_1 \right]\right\}$ is precompact in $Y_{\mathbb{C}} \times \mathbb{R}^{3}.$ Then, there exists a subsequence $\left\{\left(z^{r^n}, r^{r^{n}}, h^{r^{n}}, \theta^{r^{n}}\right)\right\}$ such that
$$
\left(z^{r^n}, \beta^{r^n}, h^{r^n}, \theta^{r^n}\right) \rightarrow\left(z^{0}, r^{0}, h^{0}, \theta^{0}\right), \quad r^n \rightarrow 0 \text { as } n \rightarrow \infty.
$$
Taking the limit of the equation $\Delta ^{-1} g_1\left(z^{r^n}, r^{r^{n}}, h^{r^{n}}, \theta^{r^{n}}\right)=0 $ as $n \rightarrow \infty$, we see that $G(z^0,\beta^0,h^0,\theta^0,0)=0 $. From Lemma \ref{z0beta0h0theta0}, we see that
$$
\left(z^{0}, r^{0}, h^{0}, \theta^{0}\right)= \left(z_{0}, r_{0}, h_{0}, \theta_{0}\right).
$$
Therefore, $ \left(z^{r}, \beta^{r}, h^{r}, \theta^{r}\right) \rightarrow\left(z_{0}, r_{0}, h_{0}, \theta_{0}\right)$ as $r  \rightarrow 0$ in the norm of $ (X)_\mathbb{C} \times \mathbb{R}^3$. This completes the proof.

\end{proof}
The following result is deduced directly from Theorem \ref{uniquesolution}.
\begin{theorem}\label{solutionvtaupsi}
Assume that $c_0>2$. Then, for each $r \in (0,\tilde r_1]$, the following equation
$$
\left\{\begin{array}{l}
{\Delta(r, {\rm i} \nu, \tau) \psi=0} \\
{\nu >0,\; \tau \geq 0,\; \psi(\neq 0) \in X_{\mathbb{C}}}
\end{array}\right.
$$
has a solution $(\nu,\tau,\psi)$, i.e. $ {\rm i}\nu \in \sigma(A_\tau(r))$ if and only if

\begin{equation}\label{taunupsi}
\nu = \nu_{r}=r h_{r},\; \psi=k \psi_{r},\; \tau=\tau_{n}=\frac{\theta_{r}+2 n \pi}{\nu_r},\; n=0,1,2, \cdots,
\end{equation}
where $\psi_r=\beta_r c_0+r z_r $, $k$ is a nonzero constant, and $z_r,\beta_r,h_r,\theta_r$ are defined as in Theorem \ref{uniquesolution}.
\end{theorem}

Then we give the following estimates to show that ${\rm i}\nu_r$ (obtained in Theorem \eqref{solutionvtaupsi}) is simple  and  the transversality condition holds.
\begin{lemma}\label{Sn}
Assume that $c_0>2$, and let
\begin{equation}\label{Sn(r)}
S_{n}(r):=\int_{\Omega} \psi_{r}^{2} d x+r \tau_{n} e^{-{\rm i} \theta_{r}} \int_{\Omega} p(x)f^\prime(u_r)  \psi_{r}^{2} d x,
\end{equation}
where $\psi_{r}$, $\tau_n$ and $\theta_r$ are defined as in Theorem \ref{solutionvtaupsi}.
Then $\lim_{r\to0} S_n(r)\ne0$ for $n =0, 1, 2, \cdots$.
\end{lemma}

\begin{proof}
From Theorems \ref{uniquesolution} and \ref{solutionvtaupsi}, we obtain that $\theta_{r} \rightarrow \theta_{0}, \tau_{n} r \rightarrow\left(\theta_0 +2 n \pi\right)/ h_{0},$ $\psi_r\rightarrow c_0 $ as $r \rightarrow 0$. This implies that
\begin{equation}
\begin{aligned}
\lim _{r \rightarrow 0} S_{n}(r) &=\int_{\Omega} c_0^2 d x+\frac{\theta_0 +2n\pi}{h_0}c_0^2 f^\prime(c_0) e^{-{\rm i} \theta_0}\int_{\Omega} p(x) d x \\
&=\left[1+\frac{\theta_0 +2n\pi}{h_0} f^\prime(c_0)e^{-{\rm i} \theta_0} \bar p \right] c_0^2 |\Omega| \neq 0.
\end{aligned}
\end{equation}
This completes the proof.
\end{proof}

Now we show that ${\rm i}\nu_r$  is simple.

\begin{theorem}\label{simpleeigenvalue}
Assume that $c_0>2$. For each $r \in (0, \check r_1]$ and $n=0,1,2, \cdots$, where $\check r_1$ is sufficiently small, $ \mu={\rm i} \nu_r$ is a simple eigenvalue of $A_{\tau_n}$.
\end{theorem}

\begin{proof}
Firstly, from Theorem \ref{solutionvtaupsi}, we have $\mathscr{N}\left[A_{\tau_{n}}(r)-{\rm i} \nu_{r}\right]=\operatorname{Span}\left[e^{{\rm i} \nu_{r} \theta} \psi_{r}\right]$, where $\theta \in [- \tau_n,0]$ and $\psi_{r}$ is defined as in Theorem \ref{solutionvtaupsi}. If $\phi_{1} \in \mathscr{N}\left[A_{\tau_{n}}(r)-{\rm i} \nu_{r}\right]^2$, i.e. $ \left[A_{\tau_{n}}(r)-{\rm i} \nu_{r}\right]^2 \phi_{1}=0$, then
$$
\left[A_{\tau_{n}}(r)-{\rm i} \nu_{r}\right] \phi_{1} \in \mathscr{N}\left[A_{\tau_{n}}(r)-{\rm i} \nu_{r}\right]=\operatorname{Span}\left[e^{{\rm i} \nu_{r} \theta} \psi_{r}\right].
$$
Therefore, there is a constant number $a$ such that
$$
\left[A_{\tau_{n}}(r)-{\rm i} \nu_{r}\right] \phi_{1}=a e^{{\rm i} \nu_{r} \theta} \psi_{r},
$$
which yields
\begin{equation}\label{dotphi}
\begin{aligned}
\dot{\phi}_{1}(\theta) &={\rm i} \nu_{r} \phi_{1}(\theta)+a e^{{\rm i} \nu_{r} \theta} \psi_{r}, \quad \theta \in\left[-\tau_{n}, 0\right], \\
\dot{\phi}_{1}(0) &=\Delta {\phi}_{1}(0)+r p(x) f^{\prime}\left(u_{r}\right) {\phi}_{1}(-\tau_n)- r \delta(x) {\phi}_{1}(0).
\end{aligned}
\end{equation}
From Eq. \eqref{dotphi}, we deduce that
\begin{equation}\label{phi}
\begin{aligned} \phi_{1}(\theta) &=\phi_{1}(0) e^{{\rm i} \nu_{r} \theta}+a \theta e^{{\rm i} \nu_{r} \theta} \psi_{r}, \\
\dot{\phi}_{1}(0) &={\rm i} \nu_{r} \phi_{1}(0)+a \psi_{r}.
\end{aligned}
\end{equation}
Then Eqs. \eqref{dotphi} and \eqref{phi} imply that
\begin{equation}\label{Deltaphi1}
\begin{aligned}
\Delta\left(r, {\rm i} \nu_{r}, \tau_{n}\right) \phi_{1}(0)=& \Delta {\phi}_{1}(0)+r p(x) f^{\prime}\left(u_{r}\right) {\phi}_{1}(0) e^{-{\rm i} \theta_r}- r \delta(x) {\phi}_{1}(0)-{\rm i} \nu_r {\phi}_{1}(0) \\
=& a \left(\psi_{r}+r \tau_n \psi_{r} p(x) f^{\prime}\left(u_{r}\right)e^{-{\rm i} \theta_r}\right).
\end{aligned}
\end{equation}
This yields
$$
\begin{aligned}
0 &=\left\langle \Delta\left(r,-{\rm i} \nu_r, \tau_{n}\right) \bar{\psi}_r, \phi_{1}(0)\right\rangle=\left\langle\bar{\psi}_r, \Delta\left(r, {\rm i} \nu_r, \tau_{n}\right) \phi_{1}(0)\right\rangle \\
&= a\left(\int_{\Omega} \psi_{r}^{2} d x+r \tau_{n} e^{-{\rm i} \theta_{r}} \int_{\Omega}p(x) f^{\prime}\left(u_{r}\right) \psi_{r}^{2}  d x\right).
\end{aligned}
$$
As a consequence of Lemma \ref{Sn}, we obtain $a = 0$ for $r\in(0,\check r_1]$, where $\check r_1$ is sufficiently small. This leads to that $ \left[A_{\tau_{n}}(r)-{\rm i} \nu_{r}\right] \phi=0$ and $\phi \in \mathscr{N}\left[A_{\tau_{n}}(r)-{\rm i} \nu_{r}\right]$. By induction we obtain
$$
\mathscr{N}\left[A_{\tau_{n}}(r)-{\rm i} \nu_r\right]^{j}=\mathscr{N}\left[A_{\tau_{n}}(r)-{\rm i} \nu_r\right], \quad j=2,3, \cdots, n=0,1,2, \cdots.
$$
Hence, $r={\rm i} \nu_r$ is a simple eigenvalue of $A_{\tau_{n}}$ for $n=0,1,2, \cdots$.
\end{proof}

Note that $\mu={\rm i} \nu_r$ is a simple eigenvalue of $A_{\tau_{n}}$, and by using the implicit function theorem we can verify that there are a neighborhood $O_{n} \times D_{n} \times H_{n} \subset \mathbb{R} \times \mathbb{C} \times X_{\mathbb{C}}$ of $\left(\tau_{n}, {\rm i} \nu_r, \psi_r\right)$ and a continuously differential function $(\mu(\tau), \psi(\tau)): O_{n} \rightarrow D_{n} \times H_{n}$ such that for each $\tau \in O_{n},$ the only eigenvalue of $A_{\tau}(r)$ in $D_{n}$ is $\mu(\tau),$ and
\begin{equation}\label{Deltapsi}
\Delta(r, \mu(\tau), \tau) \psi(\tau):=\Delta \psi(\tau)+r p(x) f'\left(u_{r}\right) \psi(\tau)e^{-\mu(\tau) \tau}- r \delta(x) \psi(\tau)-\mu(\tau) \psi(\tau)=0,
\end{equation}
where $\mu\left(\tau_{n}\right)={\rm i} \nu_r,$ and $\psi\left(\tau_{n}\right)=\psi_r$. Next, we show that the transversality condition holds.

\begin{theorem}\label{Remu}
Assume that $c_0>2$ and $r \in (0, \check r_1]$, where $\check r_1$ is sufficiently small. Then
$$
\frac{d \mathcal{R} e\left[\mu\left(\tau_{n}\right)\right]}{d \tau}>0, \quad n=0,1,2, \cdots.
$$
\end{theorem}

\begin{proof}
Differentiating Eq. \eqref{Deltapsi} with respect to $\tau$ at $\tau=\tau_n$, we obtain
\begin{equation}\label{dmu}
\begin{array}{l}
\displaystyle - {\frac{d \mu\left(\tau_{n}\right)}{d \tau}\left[ \psi_r+ r \tau_{n} p(x) f^{\prime}\left(u_{r}\right) \psi_r e^{-{\rm i} \theta_r}\right]} \\
\displaystyle {+ \Delta\left(r, {\rm i} \nu_r, \tau_{n}\right) \frac{d \psi\left(\tau_{n}\right)}{d \tau} - {\rm i} \nu_r r p(x) f^{\prime}\left(u_{r}\right) \psi_r e^{-{\rm i} \theta_r}=0}.
\end{array}
\end{equation}
Note that
\begin{equation}
\left\langle\bar{\psi}_r, \Delta\left(r, {\rm i} \nu_r, \tau_{n}\right) \frac{d\psi\left(\tau_{n}\right)}{d \tau}\right\rangle=\left\langle \Delta\left(r,-{\rm i} \nu_r, \tau_{n}\right) \bar{\psi}_r, \frac{d\psi\left(\tau_{n}\right)}{d \tau}\right\rangle= 0.
\end{equation}
Then, multiplying Eq. \eqref{dmu} by $\psi_r$ and integrating the result over $\Omega$, we have
\begin{equation}
\begin{aligned}
\frac{d \mu\left(\tau_{n}\right)}{d \tau}=& \frac{{\rm i} \nu_{r} r e^{-{\rm i} \theta_r} \int_{\Omega} p(x) f^{\prime}\left(u_{r}\right) \psi_r^{2} d x}{-\int_{\Omega} \psi_r^{2} d x- r \tau_{n} e^{-{\rm i} \theta_r} \int_{\Omega} p(x) f^{\prime}\left(u_{r}\right) \psi_r^{2} d x} \\
=&- \frac{1}{\left|S_{n}(r)\right|^{2}}\left({\rm i} \nu_r r e^{-{\rm i} \theta_r} \int_{\Omega} \bar \psi_{r}^{2} d x \int_{\Omega} p(x) f^{\prime}\left(u_{r}\right) \psi_r^{2} d x\right.\\
&\left.+{\rm i} \nu_r r^{2} \tau_{n}\left[\int_{\Omega} p(x) f^{\prime}\left(u_{r}\right)  \psi_{r}^{2} d x\right]^{2}\right).
\end{aligned}
\end{equation}
It follows from Eq. \eqref{thetah} that
$$h_0\sin \theta_0=-\bar p e^{-c_0} \frac{c_0^2 - 2{c_0}}{1 - {c_0}}.$$
This, combined with the expression of $u_r, \nu_r$, $\psi_r$ and $c_0>2$, yields
$$
\lim _{r \rightarrow 0} \frac{1}{r^{2}} \frac{d \mathcal{R} e\left[\mu\left(\tau_{n}\right)\right]}{d \tau}=\frac{1}{\lim _{r \rightarrow 0}\left|S_{n}(r)\right|^{2}}\left[(c_0^2-2c_0)e^{-2c_0}{\bar p}^2 c_0^4 |\Omega|^2\right]>0.
$$

\end{proof}

From Theorems \ref{sstability}, \ref{solutionvtaupsi}, \ref{simpleeigenvalue} and \ref{Remu}, we obtain the stability of $u_r$ and the associated Hopf bifurcation.
\begin{theorem}\label{distribution of eigenvalues}
Assume that $c_0>0$. Then model \eqref{M2} has a unique positive steady state $u_r$. Moreover, the following two statements hold for $r\in(0,\check r_1)$, where $0<\check r_1\ll 1$.
\begin{enumerate}
\item [$(i)$] If $0<c_0<2$, then $u_r$ is locally asymptotically stable for any $\tau\in[0,\infty)$.
\item [$(i)$] If $c_0>2$, then there exists a sequence $\{\tau_n\}_{n=0}^\infty$ (defined as in Theorem \ref{solutionvtaupsi}) such that
$u_r$ is locally asymptotically stable for any $\tau\in[0,\tau_0)$ and unstable for $\tau>\tau_0$, and model \eqref{M2} occurs Hopf bifurcation at $u_r$ when $\tau=\tau_n$ ($n=0,1,\dots$).
\end{enumerate}
\end{theorem}


\section{The direction of the Hopf bifurcation}
In this section, we analyze the direction of the Hopf bifurcation of Eq. \eqref{M2} by the methods in \cite{Faria2002Stability,Faria2001Normal,Faria2002Smoothness,Hassard1981Theory}. Letting $U(t)=u(\cdot, t)-u_{\lambda}, t=\tau \tilde{t}, \tau=\tau_{n}+\gamma$, and dropping the tilde sign, system \eqref{M2} can be transformed as follows:
\begin{equation}\label{directionM}
\frac{d U(t)}{d t} =\tau_{n}\Delta U(t)+\tau_{n} L_{0}\left(U_{t}\right)+J\left(U_{t}, \gamma\right),
\end{equation}
where $U_{t} \in \mathcal{C}=C([-1,0], Y)$, and
$$
\begin{aligned}
L_{0}\left(U_{t}\right)=&-r \delta(x) U(t)+ r p(x) f^{\prime}\left(u_{r}\right) U(t-1), \\
J\left(U_{t}, \gamma\right)=&\gamma \Delta U(t)+\gamma L_{0}\left(U_{t}\right)+ \left(\gamma+\tau_{n}\right) r p(x) \\
&\times \left[\frac{f^{\prime \prime}\left(u_r\right)}{2} U^{2}(t-1)+\frac{f^{\prime \prime \prime}\left(u_r\right)}{3 !} U^{3}(t-1)+\mathcal{O}\left(U^{4}(t-1)\right)\right].
\end{aligned}
$$
Then $\gamma = 0$ is the Hopf bifurcation value of Eq. \eqref{directionM}.

Define by $\mathcal{A}_{{\tau}_n}$ the infinitesimal generator of the linearized equation
\begin{equation}
\frac{d U(t)}{d t} =\tau_{n}\Delta U(t)+\tau_{n} L_{0}\left(U_{t}\right).
\end{equation}
It follows from \cite[Chapter 3]{Wu1996Theory} that
$$
\begin{aligned}
\mathcal{A}_{\tau_{n}} \Psi=& \dot{\Psi}, \\
\mathscr{D}\left(\mathcal{A}_{\tau_{n}}\right)=&\left\{\Psi \in \mathcal{C}_{\mathbb{C}} \cap \mathcal{C}_{\mathbb{C}}^{1}: \Psi(0) \in X_{\mathbb{C}}, \dot{\Psi}(0)=\tau_{n}\Delta \Psi(0)\right.\\
&\left.-r \tau_{n}\delta(x) \Psi(0)+ r\tau_{n} p(x) f^{\prime}\left(u_{r}\right)\Psi(-1)\right\},
\end{aligned}
$$
where $C_{\mathbb{C}}^{1}=C^{1}\left([-1, 0], Y_{\mathbb{C}}\right),$ and Eq. \eqref{directionM} can be written in the following abstract form
\begin{equation}\label{absfor}
\frac{d U_t}{d t} =\mathcal{A}_{{\tau}_n} U_t+X_{0} J\left(U_{t}, \gamma\right),
\end{equation}
and
$$
X_{0}(\theta)=\left\{\begin{array}{ll}{0,} & {\theta \in[-1,0)}, \\
{I,} & {\theta=0}.\end{array}\right.
$$
Clearly, $\mathcal{A}_{\tau_{n}}$ has only one pair of purely imaginary eigenvalues $\pm {\rm i} \nu_r \tau_n$ which are simple, and the corresponding eigenfunction with respect to ${\rm i} \nu_r \tau_n$(respectively, $- {\rm i} \nu_r \tau_n$) is $\psi_{r} e^{{\rm i} \nu_{r} \tau_n \theta}$ (respectively, $\overline\psi_{r} e^{-{\rm i} \nu_{r} \tau_n \theta}$) for $\theta \in [-1, 0]$, where $\psi_{r}$ is defined as in Theorem \ref{solutionvtaupsi}.

It follows from \cite{Faria2002Stability,SuWeiShi2012} that we introduce the formal duality $\langle\langle \cdot, \cdot \rangle\rangle$ in $\mathcal {C}$ by
\begin{equation}\label{fdual}
\langle\langle\tilde{\Psi}, \Psi\rangle\rangle=\langle\tilde{\Psi}(0), \Psi(0)\rangle+r\tau_{n} \int_{-1}^{0}\left\langle\tilde{\Psi}(s+1), p(x) f^{\prime}\left(u_{r}\right) \Psi(s)\right\rangle d s,
\end{equation}
for $\Psi \in \mathcal{C}_{\mathbb{C}}$ and $\tilde{\Psi} \in \mathcal{C}_{\mathbb{C}}^{*}:=C\left([0,1], Y_{\mathbb{C}}\right)$. Similar to \cite[Chapter 6]{Hale1977Theory} (see also \cite[Lemma 3.1]{Chen2018Hopf}), we can obtain the formal adjoint operator $\mathcal{A}_{{\tau}_n}^{*}$ of $\mathcal{A}_{{\tau}_n}$ with respect to the formal duality \eqref{fdual}. Here we omit the proof.

\begin{lemma}
The formal adjoint operator $\mathcal{A}_{{\tau}_n}^{*}$ of $\mathcal{A}_{{\tau}_n}$ is defined by
$$
\mathcal{A}_{\tau_{n}}^{*} \tilde{\Psi}(s)=-\dot{\tilde{\Psi}}(s), \quad s \in [0,1]
$$
and

$$
\begin{aligned}
\mathscr{D}\left(\mathcal{A}_{\tau_{n}^{*}}\right)=&\left\{\tilde \Psi \in \mathcal{C}^*_{\mathbb{C}} \cap (\mathcal{C}^*_{\mathbb{C}})^{1}: \tilde\Psi(0) \in X_{\mathbb{C}}, \dot{\tilde{\Psi}}(0)=\tau_{n}\Delta {\tilde{\Psi}}(0)\right.\\
&\left.-r \tau_{n}\delta(x) {\tilde{\Psi}}(0)+ r\tau_{n} p(x) f^{\prime}\left(u_{r}\right){\tilde{\Psi}}(1)\right\},
\end{aligned}
$$
where $\left(\mathcal{C}_{\mathbb{C}}^{*}\right)^{1}=C^{1}\left([0,1], Y_{\mathbb{C}}\right) .$ Moreover, $\mathcal{A}_{\tau_{n}}^{*}$ and $\mathcal{A}_{\tau_{n}}$ satisfy
$$
\langle\langle\mathcal{A}_{\tau_{n}}^{*} \tilde{\Psi}, \Psi\rangle\rangle=\langle\langle\tilde{\Psi}, \mathcal{A}_{\tau_{n}} \Psi\rangle\rangle \text { for } \Psi \in \mathscr{D}\left(\mathcal{A}_{\tau_{n}}\right) \text { and } \tilde{\Psi} \in \mathscr{D}\left(\mathcal{A}_{\tau_{n}}^{*}\right).
$$
\end{lemma}


Similar to Theorem \ref{distribution of eigenvalues}, the operator $\mathcal{A}_{{\tau}_n}^{*}$ has only one pair of purely imaginary eigenvalues  $\pm {\rm i} \nu_r \tau_n$ which are simple, and the associated eigenfunction with respect to $- {\rm i} \nu_r \tau_n$ (respectively, ${\rm i} \nu_r \tau_n$) is $\overline \psi_{r} e^{{\rm i}\nu_{r} \tau_n s}$ (respectively, $\psi_{r} e^{-{\rm i}\nu_{r} \tau_n s}$) for $s \in [0, 1]$, where $\psi_{r}$ is defined as in Theorem \ref{solutionvtaupsi}. Then the center subspace of Eq. \eqref{absfor} is $P=\operatorname{Span}\{{p(\theta),\overline p(\theta)}\}$, where $p(\theta)= \psi_{r} e^{{\rm i} \nu_{r} \tau_n \theta}$, and the formal adjoint subspace of $P$ is $P^*=\operatorname{Span}\{{q(s),\overline q(s)}\}$, where $q(s)=\overline \psi_{r} e^{{\rm i}\nu_{r} \tau_n s}$. Let $\displaystyle \Phi_{p}=(p(\theta), \overline {p}(\theta)), \Psi_{P}=\frac{1}{\overline{S_{n}}(r)}(q(s), \overline {q}(s))^{T}$, we can check that $\langle\langle{\Psi}_p, \Phi_{p}\rangle\rangle= I$ directly, where $I$ is the identity matrix in $\mathbb{R}^{2 \times 2}$.
Therefore, $\mathcal {C}_{\mathbb{C}}$ can be decomposed as $\mathcal {C}_{\mathbb{C}}=P\oplus Q$, where
$$
Q=\left\{\Psi \in \mathcal{C}_{\mathbb{C}}:\langle\langle\tilde{\Psi}, \Psi\rangle\rangle= 0 \text { for all } \tilde{\Psi} \in P^{*}\right\}.
$$

Let $\gamma = 0$ in Eq. \eqref{absfor}, and one could obtain the center manifold
\begin{equation}
w(z, \overline{z})=w_{20}(\theta) \frac{z^{2}}{2}+w_{11}(\theta) z \overline {z}+w_{02}(\theta) \frac{\overline {z}^{2}}{2}+O\left(|z|^{3}\right)
\end{equation}
with the range in $Q$. Then the flow of Eq. \eqref{directionM} on the center manifold can be written as:
$$
U_{t}=\Phi_{p} \cdot(z(t), \overline {z}(t))^{T}+w(z(t), \overline {z}(t)),
$$
and
\begin{equation}
\dot{z}(t) =\frac{d}{d t}\langle\langle q(s), U_{t}\rangle\rangle
={\rm i} \nu_r \tau_{n} z(t)+g(z,\overline z),
\end{equation}
where
\begin{equation}
\begin{aligned}
g(z, \overline{z})
&=\frac{1}{S_{n}(r)}\left\langle q(0), J\left(\Phi_{p}(z(t), \overline{z}(t))^{T}+w(z(t),\overline {z}(t)), 0\right)\right\rangle\\
&=\sum_{2 \leq i+j \leq 3} \frac{g_{i j}}{i ! j !} z^{i} \overline{z}^{j}+O\left(|z|^{4}\right).
\end{aligned}
\end{equation}
A direct computation implies that
\begin{equation}\label{g}
\begin{aligned}
g_{20}=&\frac{ r \tau_{n}}{S_{n}(r)} e^{-2 {\rm i} \nu_{r} \tau_{n}} \int_{\Omega} p(x) f''(u_r) \psi_{r}^{3} d x, \\
g_{11}=&\frac{ r \tau_{n}}{S_{n}(r)} \int_{\Omega} p(x) f''(u_r) \psi_{r}\left|\psi_{r}\right|^{2} d x, \\
g_{02}=&\frac{ r \tau_{n}}{S_{n}(r)} e^{2 {\rm i} \nu_{r} \tau_{n}}\int_{\Omega} p(x) f''(u_r) \psi_{r} \overline {\psi}_{r}^{2} d x, \\
g_{21}=&\frac{2 r \tau_{n}}{S_{n}(r)}  e^{- {\rm i} \nu_{r} \tau_{n}} \int_{\Omega} p(x) f''(u_r) \psi_{r}^{2} w_{11}(-1) dx \\ &+\frac{r \tau_{n}}{S_{n}(r)} e^{{\rm i} \nu_{r} \tau_{n}} \int_{\Omega} p(x) f''(u_r)\left|\psi_{r}\right|^{2} w_{20}(-1)dx \\
&+\frac{r \tau_{n}}{S_{n}(r)} e^{-{\rm i} \nu_{r} \tau_{n}} \int_{\Omega} p(x) f'''(u_r) \psi_{r}^{2}|\psi_{r}|^2 d x .
\end{aligned}
\end{equation}

We need to calculate $w_{20}(\theta)$ and $w_{11}(\theta)$ to solve $g_{21}$. From \cite{Hassard1981Theory}, we obtain that $w_{20}(\theta)$ and $w_{11}(\theta)$ satisfy the following equalities,
\begin{equation}\label{2 equalities}
\left\{\begin{array}{l}
{\left(2 {\rm i} \nu_{r} \tau_{n}-\mathcal{A}_{\tau_{n}}\right) w_{20}=H_{20}}, \\
{-\mathcal{A}_{\tau_{n}} w_{11}=H_{11}}.
\end{array}\right.
\end{equation}
Note that, for $- 1 \le \theta  < 0$,
\begin{equation}\label{H20H11}
\begin{array}{l}
{H_{20}(\theta)=-\left(g_{20} p(\theta)+\overline{g}_{02} \overline{p}(\theta)\right)}, \\
{H_{11}(\theta)=-\left(g_{11} p(\theta)+\overline{g}_{11} \overline{p}(\theta)\right)},
\end{array}
\end{equation}
and for $\theta=0$,
\begin{equation}\label{the0}
\begin{aligned}
H_{20}(0)=-\left(g_{20} p(0)+\overline{g}_{02} \overline{p}(0)\right)+ r \tau_{n} e^{-2{\rm i} \nu_{r} \tau_{n}} p(x) f''(u_r) \psi_r^{2}, \\
H_{11}(0)=-\left(g_{11} p(0)+\overline{g}_{11} \overline{p}(0)\right)+ r \tau_{n} p(x) f''(u_r) \left|\psi_r\right|^{2}.
\end{aligned}
\end{equation}
Then we see from Eq. \eqref{2 equalities} and \eqref{H20H11} that $w_{20}(\theta)$ and $w_{11}(\theta)$ can be expressed as
\begin{equation}
w_{20}(\theta)=\frac{{\rm i} g_{20}}{\nu_{r} \tau_{n}} p(\theta)+\frac{{\rm i} \overline{g}_{02}}{3 \nu_{r} \tau_{n}} \overline {p}(\theta)+E e^{2 {\rm i} \nu_{r} \tau_{n} \theta},
\end{equation}
and
\begin{equation}
w_{11}(\theta)=-\frac{{\rm i} g_{11}}{\nu_{r} \tau_{n}} p(\theta)+\frac{{\rm i} \overline {g}_{11}}{\nu_{r} \tau_{n}} \overline {p}(\theta)+F.
\end{equation}
From Eqs. \eqref{2 equalities} and \eqref{the0} and the definition of $\mathcal{A}_{\tau_{n}}$, we find that $E$ satisfies
\begin{equation}
\left.\left(2 {\rm i} \nu_{r} \tau_{n}-\mathcal{A}_{\tau_{n}}\right) E e^{2 {\rm i} \nu_{r} \tau_{n} \theta}\right|_{\theta=0}=r \tau_{n} e^{-2{\rm i} \nu_{r} \tau_{n}}  p(x) f''(u_r)  \psi_r^{2},
\end{equation}
or equivalently,
\begin{equation}\label{E}
\Delta\left(r, 2 {\rm i} \nu_{r}, \tau_{n}\right) E=- r e^{-2{\rm i} \nu_{r} \tau_{n}}  p(x) f''(u_r) \psi_r^{2},
\end{equation}
where $\Delta(r, \mu, \tau)$ is defined in Eq. \eqref{Deltamu}. It follows from Theorem \ref{solutionvtaupsi} that $2 {\rm i} \nu_r$ is not the eigenvalue of $\mathcal{A}_{\tau_{n}}$ for $r \in (0, \check r_1]$, where $0<\check r_1\ll1$, and consequently
$$
E=- r e^{-2{\rm i} \nu_{r} \tau_{n}} \Delta\left(r, 2{\rm i} \nu_{r}, \tau_{n}\right)^{-1}\left(p(x) f''(u_r)\psi_r^{2}\right).
$$
Similarly, we see that $F$ satisfies
\begin{equation}\label{F}
F=-r \Delta\left(r, 0, \tau_{n}\right)^{-1}\left( p(x) f''(u_r) \left|\psi_r\right|^{2}\right).
\end{equation}
In the following, functions $E$ and $F$ could be determined.
\begin{lemma}\label{EF}
Assume that $E$ and $F$ satisfy \eqref{E} and \eqref{F}, respectively. Then
\begin{equation}\label{estimateEF}
E=k_r c_0 +\eta _r ,\;\; F =l_r c_0+ \tilde \eta _r.
\end{equation}
Here $c_0$ is defined as in \eqref{c0}, $\eta _r$ and $\tilde \eta _r$ satisfy
$$
\eta _r,\tilde \eta _r\in X_1, \;\; \lim _{r \rightarrow 0}\left\|\eta_r \right\|_{X_{\mathbb{C}}}=0, \;\; \lim _{r \rightarrow 0}\left\|\tilde{\eta}_r\right\|_{X_{\mathbb{C}}}=0,
$$
where $X_1$ is defined as in \eqref{X1Y1},
and the constants $k_r$ and $l_r$ satisfies
\begin{equation}\label{krlr}
\lim_{r\rightarrow 0 } k_r = \frac{-e^{-2{\rm i}\theta_0 }\overline p f''(c_0)c_0}{e^{-2{\rm i}\theta_0 }\overline p f'(c_0)-\overline \delta-2{\rm i} h_0 },\;\;\; \lim_{r\rightarrow 0 } l_r = \frac{-\overline p f''(c_0)c_0}{\overline p f'(c_0)-\overline \delta},
\end{equation}
where $\theta_0$ and $h_0$ are defined as in \eqref{cos0h0}.
\end{lemma}

\begin{proof}
We first prove the estimate for $E$.
Substituting $E$ (defined as in Eq. \eqref{estimateEF}) into Eq. \eqref{E}, we see that
\begin{equation}\label{L_reta_r}
\begin{aligned}
\Delta  & \eta_r+r e^{-2 {\rm i} \nu_r \tau_{n}} p(x) f'(u_r) \left(k_r c_0+\eta_r\right)\\
&-r\delta(x)\left(k_r c_0+\eta_r\right) -2 {\rm i} \nu_r \left(k_r c_0+\eta_r\right)\\
=& -r e^{-2 {\rm i} \nu_r \tau_{n}} p(x) f''(u_r) \psi_r^{2}.
\end{aligned}
\end{equation}
Integrating \eqref{L_reta_r} over $\Omega$, one could easily obtain
\begin{equation}\label{L_reta_rintegrating}
\begin{aligned}
k_r & \left( r e^{-2 {\rm i} \nu_r \tau_{n}} c_0 \int_{\Omega} p(x) f'(u_r) dx -r c_0\int_{\Omega} \delta(x) dx -2 {\rm i} \nu_r c_0 |\Omega| \right)\\
=& -r e^{-2 {\rm i} \nu_r \tau_{n}}\int_{\Omega} p(x) f'(u_r) \eta_r dx +2 {\rm i} \nu_r \int_{\Omega} \eta_r dx +r  \int_{\Omega}  \delta(x) \eta_r dx \\
&- r e^{-2{\rm i} \nu_{r} \tau_{n}} \int_{\Omega} p(x) f''(u_r) \psi_r^{2} dx.
\end{aligned}
\end{equation}
Then multiplying Eq. \eqref{L_reta_r} by $\overline \eta_r$, and integrating the result over $\Omega$, we have
\begin{equation}\label{etarLretaur}
\begin{aligned}
\langle\eta_r,& \Delta \eta_r \rangle + r k_r e^{-2 {\rm i} \nu_r \tau_{n}} c_0 \int_{\Omega} p(x) f'(u_r)\overline \eta_r dx-r k_r c_0 \int_{\Omega} \delta(x) \overline \eta_r dx -2 {\rm i} \nu_r k_r c_0 \int_{\Omega} \overline \eta_r dx \\
=& -r e^{-2 {\rm i} \nu_r \tau_{n}}\int_{\Omega} p(x) f'(u_r) |\eta_r|^2 dx +r \int_{\Omega} \delta(x)|\eta_r|^2 dx +2 {\rm i} \nu_r\int_{\Omega} |\eta_r|^2  dx \\
&- r e^{-2{\rm i} \nu_{r} \tau_{n}} \int_{\Omega} p(x) f''(u_r)\overline \eta_r \psi_r^{2} dx.
\end{aligned}
\end{equation}
From the expression of $\nu_r$, $u_r$, $\psi_r$ and $\tau_r$ (see Eqs. \eqref{cos0h0} and \eqref{taunupsi}), we have
\begin{equation}\label{estimate when r0}
\psi_r \rightarrow c_0,\;\; u_r \rightarrow c_0,\;\;\nu_r/r \rightarrow h_0,\;\; \nu_r\tau_{n} \rightarrow\left(\theta_0 +2 n \pi\right) \;\; as\;\; r \rightarrow 0.
\end{equation}
Hence, it follows from \eqref{L_reta_rintegrating} and \eqref{estimate when r0} that there exist constants $\tilde r>0$ and $M_0,M_1>0$ such that for any $r \in (0,\tilde r)$, $|k_r|\le M_0 \|\eta_r \|_{Y_\mathbb{C}}+M_1.$ Then from Eq.\eqref{etarLretaur} and Eq.\eqref{estimate when r0}, we obtain that there exist constants $M_2,M_3>0$ such that for any $r \in (0,\tilde r)$,
$$|r_2|\cdot \|\eta_r \|_{Y_\mathbb{C}}^2 \le r M_2 \|\eta_r \|_{Y_\mathbb{C}}^2+r M_3 \|\eta_r \|_{Y_\mathbb{C}},  $$
where $r_{2}$ (defined as in Lemma \ref{r2}) is the second eigenvalue of $-\Delta,$ and consequently, $\lim _{r \rightarrow 0}\left\|\eta_r \right\|_{Y_{\mathbb{C}}}=0$. This, combined with Eq. \eqref{L_reta_rintegrating}, implies $k_r$ satisfies \eqref{krlr}.
Then we see from Eq. \eqref{L_reta_r} that $\lim _{r \rightarrow 0}\left\|\eta_r \right\|_{X_{\mathbb{C}}}=0$.

Now we consider $F$. Similarly, substituting $F$ (defined as in Eq. \eqref{estimateEF}) into Eq. \eqref{F}, we obtain that
\begin{equation}\label{L1_reta_r}
\begin{aligned}
\Delta  & \tilde\eta_r+r p(x) f'(u_r) \left(l_r c_0+\tilde\eta_r\right)-r\delta(x)\left(l_r c_0+\tilde\eta_r\right) \\
=& -r p(x) f''(u_r) |\psi_r|^{2}.
\end{aligned}
\end{equation}
Then by using the similar arguments as that for $E$, we see that
$$\lim_{r\rightarrow 0 } l_r = \frac{-\overline p f''(c_0)c_0}{\overline p f'(c_0)-\overline \delta}, $$
and $\lim _{r \rightarrow 0}\left\|\eta_r \right\|_{X_{\mathbb{C}}}=0$.
This complete the proof.
\end{proof}

Note from \cite[Chapter 1]{Hassard1981Theory} that the following quantity determine the direction and stability of bifurcating periodic orbits:
\begin{equation}\label{C1}
C_{1}(0)=\frac{\rm i}{2 \nu_r \tau_{n}}\left(g_{11} g_{20}-2\left|g_{11}\right|^{2}-\frac{\left|g_{02}\right|^{2}}{3}\right)+\frac{g_{21}}{2}.
\end{equation}
Then we have the following result, see the Appendix for the proof.
\begin{proposition}\label{rec0}
Assume that $c_0>2$. Then $\lim_{r \rightarrow 0}\mathcal{R}e [C_{1}(0)]<0$.
\end{proposition}
Therefore, we see from proposition \ref{rec0} and \cite[Chapter 1]{Hassard1981Theory} that:

\begin{theorem}\label{stabliHopf}
Assume that $c_0>2$. For $r\in(0,\check r_1)$ and $0<\check r_1\ll 1$, let $\tau_{n}(r)$ be the Hopf bifurcation points of Eq. \eqref{M2} defined as in Theorem \ref{solutionvtaupsi}. Then for each $n \in \mathbb{N} \cup\{0\},$ the direction of the Hopf bifurcation at $\tau=\tau_{n}$ is forward, that is, the bifurcating periodic solutions exist for $ \tau>\tau_n$. Moveover, the bifurcating periodic solution from $\tau=\tau_{0}$ is orbitally asymptotically stable.
\end{theorem}

\section{Discussion}
In Sections 2 and 3, we consider the stability and Hopf bifurcation for model \eqref{M2}.
Note that model \eqref{M2} is equivalent to \eqref{M1}, and $\tau=d\hat\tau$ and $r=1/d$, where parameters $d, \hat\tau$ are in model \eqref{M1}, and parameters $r,\tau$ are in model \eqref{M2}. Then we see from Theorems \ref{distribution of eigenvalues} and \ref{stabliHopf} that:
\begin{proposition}\label{equivthm37}
Assume that $d,a>0$ and $\overline p>\overline \delta$, where $\overline p$ and $\overline \delta$ are defined as in \eqref{c0}. Then model \eqref{M1} admits a unique positive steady state $u^d$, and the following statements hold for $d\in(d_1,\infty]$, where $d_1$ is sufficiently large.
\begin{enumerate}
\item [$(i)$] If $1<\ds\f{\overline p}{\overline\delta}<e^2$, then $ u^d$ is locally asymptotically stable for any $\hat\tau\ge0$.
\item [$(ii)$] If $\ds\f{\overline p}{\overline\delta}>e^2$, then there exists a sequence $\{\hat\tau_n\}_{n=0}^\infty$ such that
$u^d$ is locally asymptotically stable for any $\tau\in[0,\hat\tau_0)$ and unstable for $\tau>\hat\tau_0$, and model \eqref{M2} occurs Hopf bifurcation at $u^d$ when $\hat\tau=\hat\tau_n$ ($n=0,1,\dots$). Moveover, for each $n \in \mathbb{N} \cup\{0\}$, the direction of the Hopf bifurcation at $\hat\tau=\hat\tau_{n}$ is forward, that is, the bifurcating periodic solutions exist for $\hat \tau>\hat\tau_n$, and the bifurcating periodic solution from $\hat\tau=\hat\tau_{0}$ is orbitally asymptotically stable.
\end{enumerate}
\end{proposition}

Note that $\tau=d\hat \tau$ and $r=1/d$, where parameters $d, \hat \tau$ are in \eqref{M1}, and parameters $r,\tau$ are in \eqref{M2}. Then we see from Lemma \ref{z0beta0h0theta0}
and Theorem \ref{solutionvtaupsi} that
\begin{equation*}
\lim_{d\to\infty} \hat\tau_0=\lim_{d\to\infty}\displaystyle\frac{\tau_0}{d}=\lim_{r\to0}r\tau_0=\check \tau_0:=\displaystyle\frac{\theta_0}{h_0}.
\end{equation*}
Here $\check \tau_0$ is the first Hopf  bifurcation value for the following model:
\begin{equation}\label{average}
u'=\overline p u(t-\check \tau)e^{-au(t-\check \tau)}-\overline \delta u.
\end{equation}
Therefore, when the diffusion rate tends to infinity,
the first Hopf bifurcation value of model \eqref{M1} tends to that of the \lq\lq average\rq\rq~DDE model \eqref{average}.
Moreover, By using the similar arguments as in the Appendix, we see that
$$\lim_{r \rightarrow 0}\mathcal{R}e [C_{1}(0)]=c_0^2  \mathcal{R}e[\check C_{1}(0)]<0,$$
where $\mathcal{R}e [C_{1}(0)]$ is the quantity which determine the direction of Hopf bifurcation for  model \eqref{M2}, and $\mathcal{R}e[\check C_{1}(0)]$ is that
for model \eqref{average}. Therefore, for model \eqref{average}, the direction of the Hopf bifurcation is also forward, and the bifurcating periodic
solution from the first Hopf bifurcation value is also orbitally asymptotically stable. This result improves the earlier result in \cite{WeiLi}.

Finally, we give some numerical simulations to demonstrate our theoretical results for Eq. \eqref{M1}. We show that when $1<c_0<2$, the solution converges to the unique positive steady state for any $\tau\ge0$, see Fig. \ref{Nichoisonstable12}. For $c_0>2$, we show that large delay $\tau$ could make the positive steady state unstable through Hopf bifurcation, and the solution converges to a positive periodic solution, see Fig. \ref{Nichoisonstablebifurcation}.

\begin{figure}[t]
\begin{center}
\includegraphics[width=0.5\textwidth]{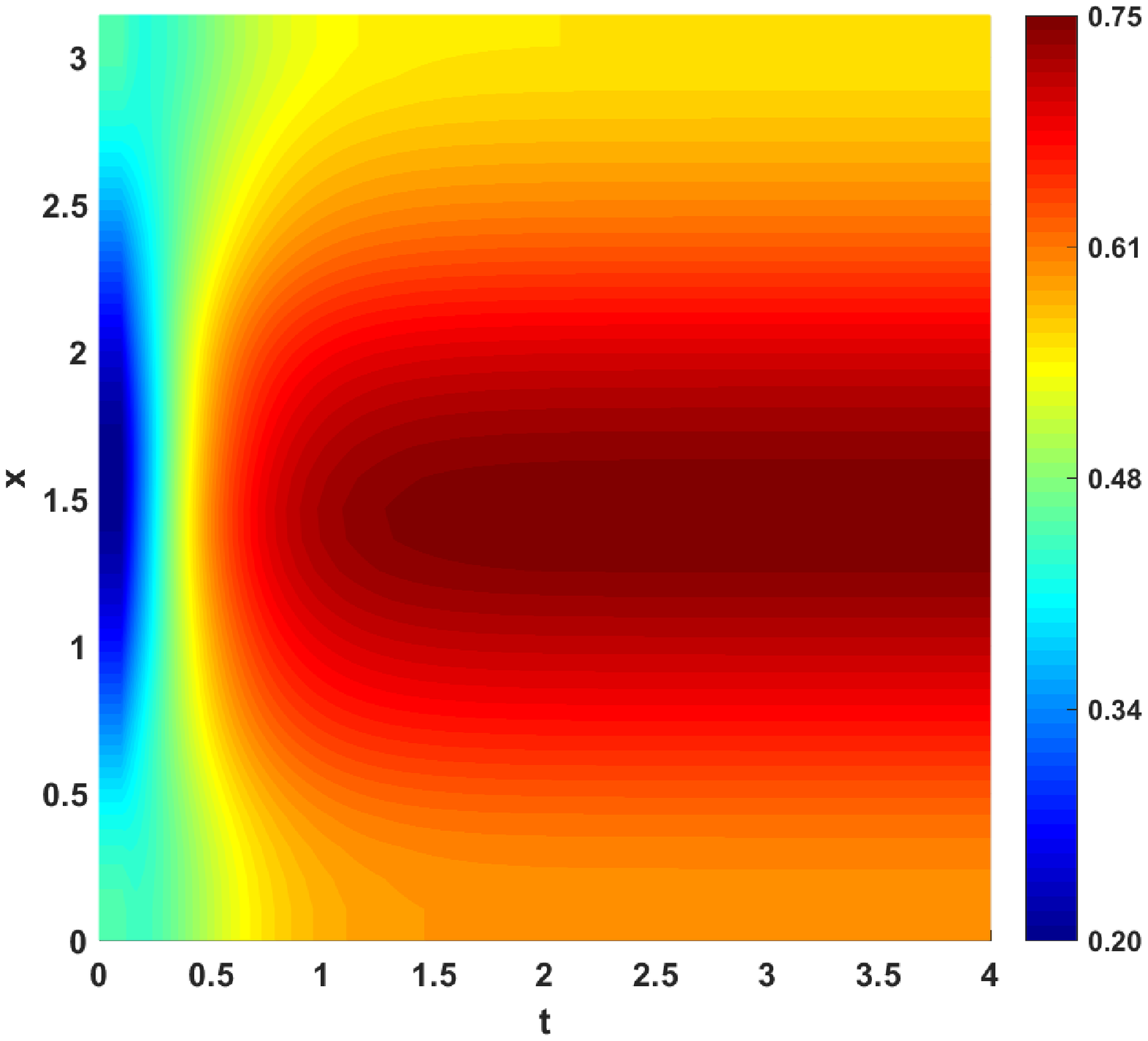}
\includegraphics[width=0.5\textwidth]{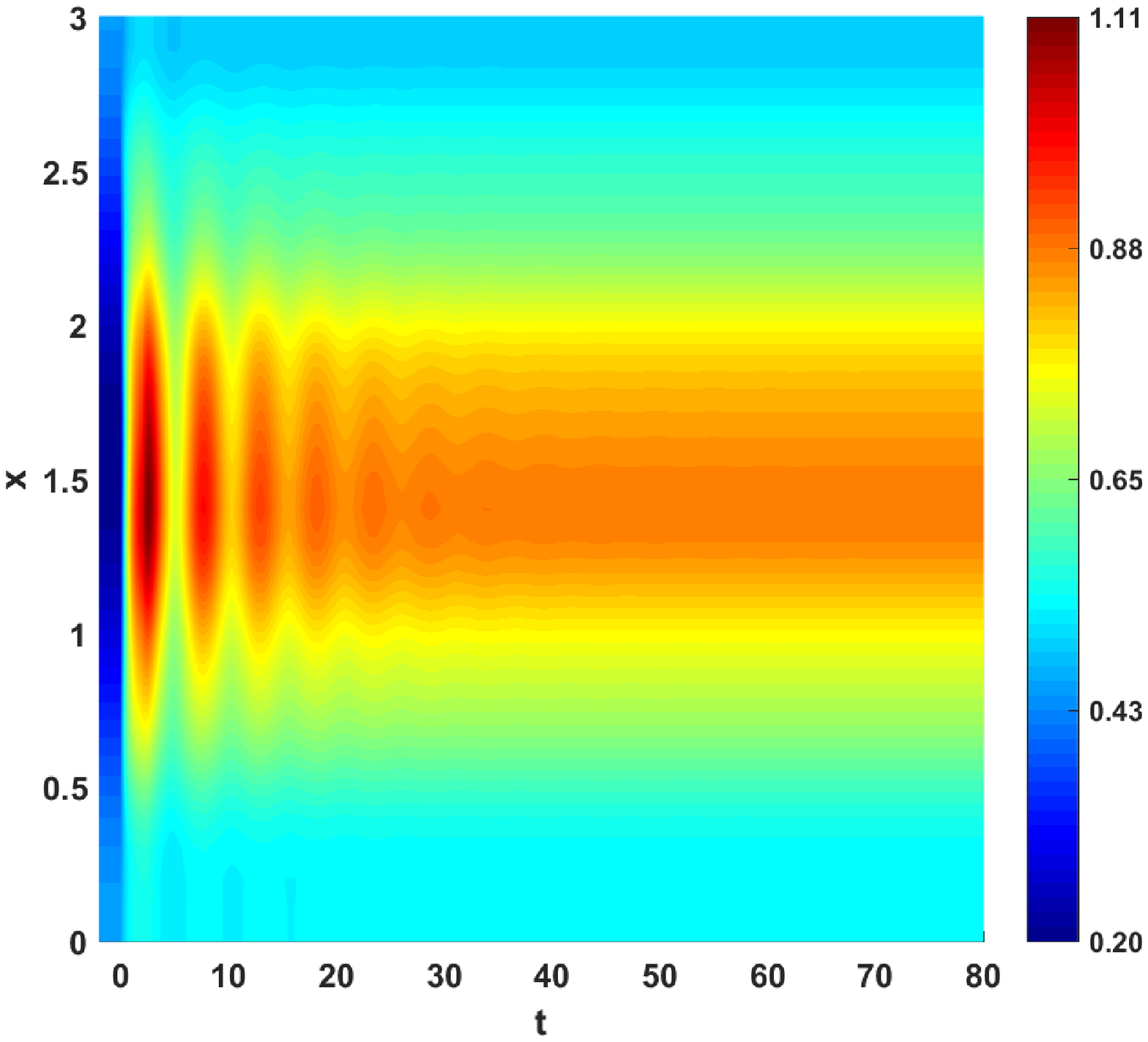}
\caption{The case of $0<c_0<2$. Here $\Omega=(0,3), d=0.1,$ and $a=2.5$, $p(x)=10+\sin x$, $\delta(x)=2+\cos 0.2x$ and $c_0=1.2880$. (Left): $\tau=0$; (Right) $\tau=2$.}
\label{Nichoisonstable12}
\end{center}
\end{figure}
\begin{figure}[t]
\begin{center}
\includegraphics[width=0.5\textwidth]{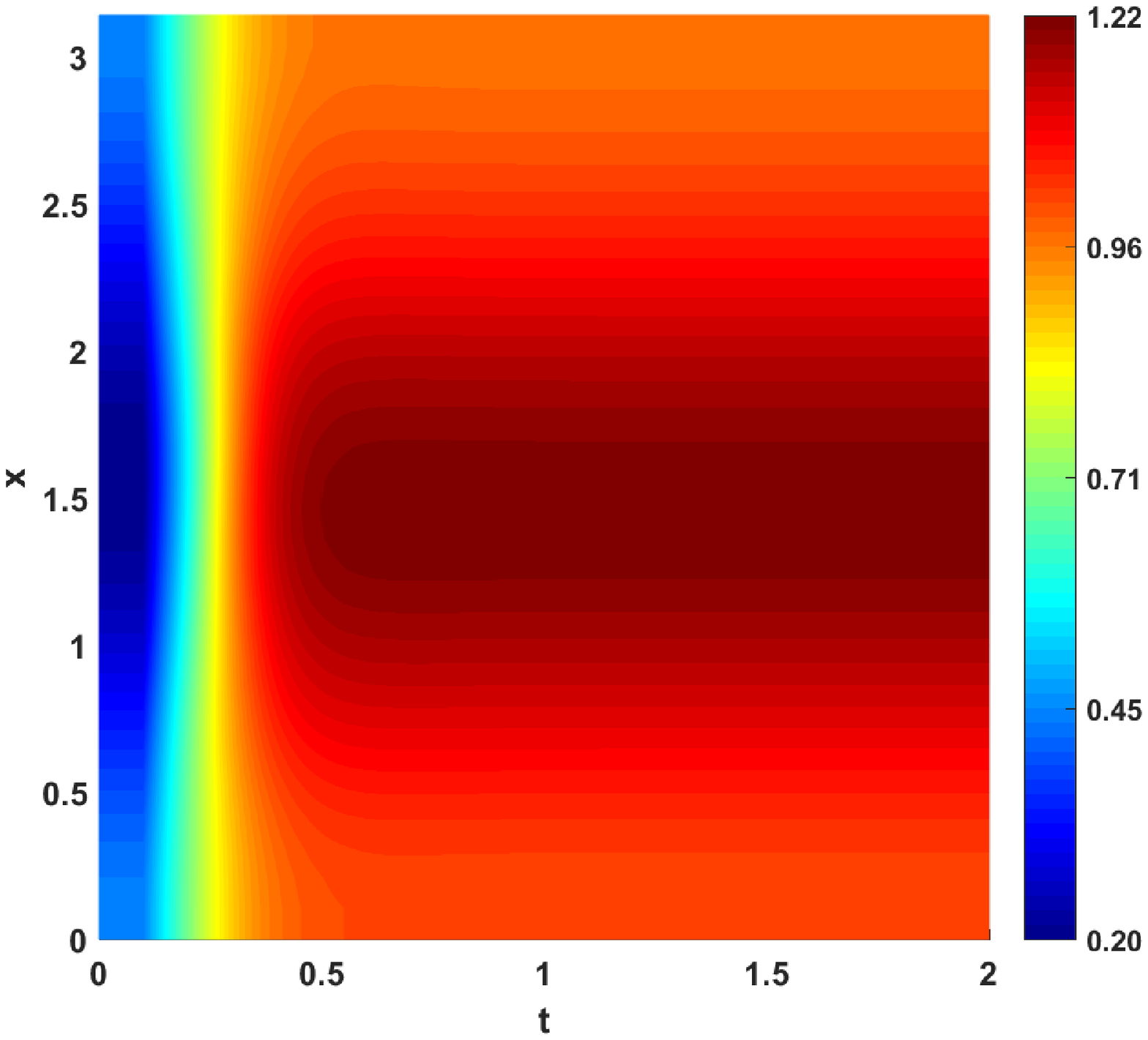}
\includegraphics[width=0.5\textwidth]{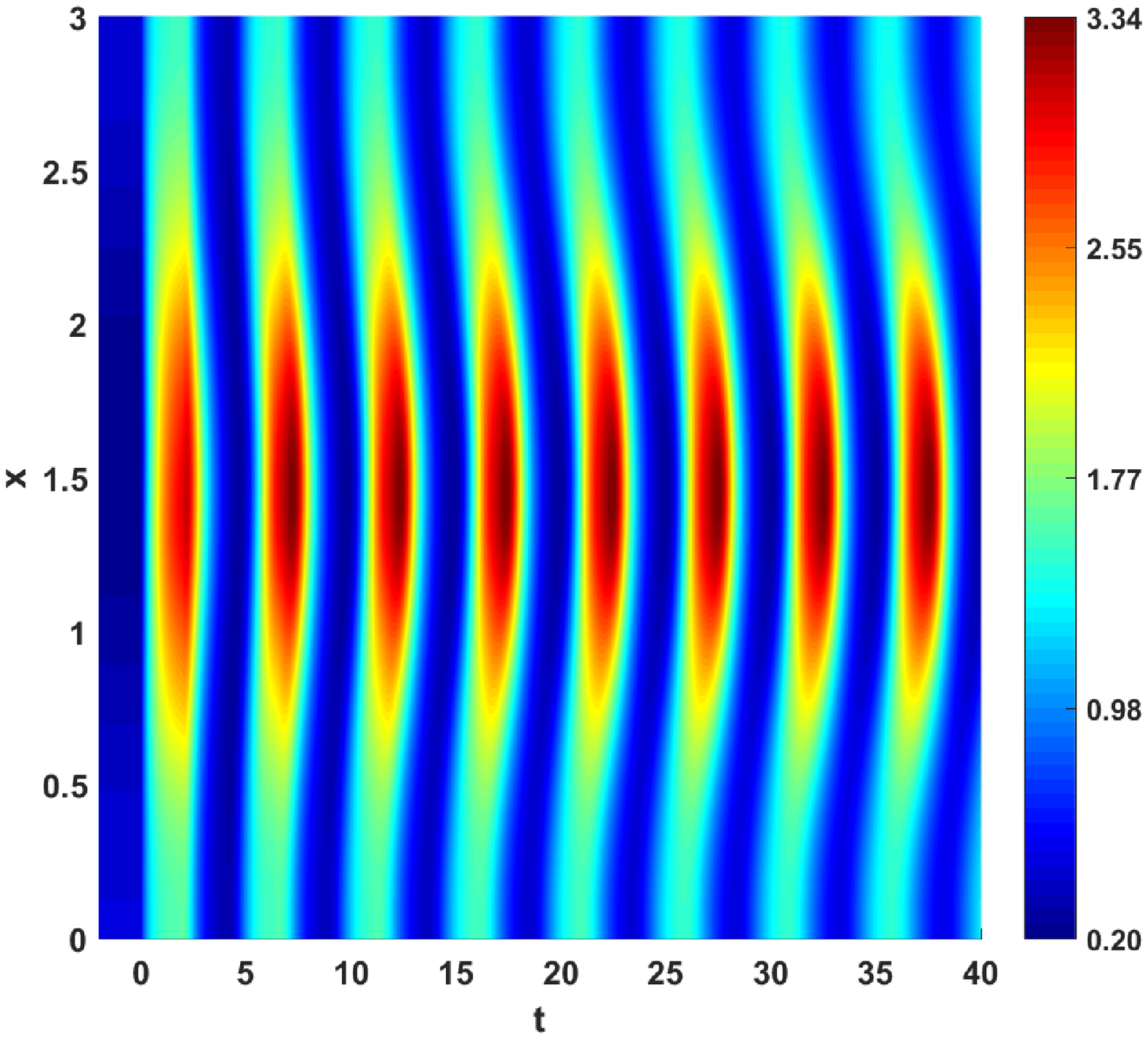}
\caption{The case of $c_0>2$. Here $\Omega=(0,3), d=0.1,$ and $a=2.5$, $p(x)=30+\sin x$, $\delta(x)=2+\cos 0.2x$ and $c_0=2.3443$. (Left) $\tau=0$; (Right) $\tau=2$.}
\label{Nichoisonstablebifurcation}
\end{center}
\end{figure}

\appendix
\section{Appendix}
{\bf The proof of Proposition \ref{rec0}:}
\begin{proof}
It follows from Lemmas \ref{z0beta0h0theta0}, \ref{EF} and Theorem \ref{solutionvtaupsi} that
\begin{equation}\label{A1}
\begin{aligned}
\cos2\theta_0=&\frac{1-c_0^2+2c_0}{(1-c_0)^2},\;\;\; \sin2\theta_0=-\frac{2{\sqrt {c_0^2 - 2{c_0}} }}{(1 - {c_0})^2},\;\;\; h_0= \bar\delta \sqrt{c_0^2 - 2{c_0}},\\
&\lim_{r \rightarrow 0}r \tau_n=\frac{\theta_0+2n\pi}{h_0},\;\;\;\lim_{r \rightarrow 0} \nu_r \tau_n=\theta_0+2n\pi,\\
\lim_{r\rightarrow 0 } E &= \frac{-e^{-2{\rm i}\theta_0 }\overline p f''(c_0)c_0^2}{e^{-2{\rm i}\theta_0 }\overline p f'(c_0)-\overline \delta-2{\rm i} h_0 },\;\;\; \lim_{r\rightarrow 0 } F = \frac{-\overline p f''(c_0)c_0^2}{\overline p f'(c_0)-\overline \delta}.
\end{aligned}
\end{equation}
Since $\lim_{r \rightarrow 0 } \psi(x)=\lim_{r \rightarrow 0 } \overline\psi(x)=c_0$, we see from Eq. \eqref{g} that
\begin{equation}\label{g201102}
\begin{aligned}
\lim_{r \rightarrow 0 } g_{20}=&\lim_{r \rightarrow 0 }\frac{ r \tau_{n}}{S_{n}(r)} e^{-2 {\rm i} \nu_r \tau_n} \overline p |\Omega| f''(c_0) c_0^3=\lim_{r \rightarrow 0 } g_{11} e^{-2 {\rm i} \nu_r \tau_n},\\
\lim_{r \rightarrow 0 } g_{11}=&\lim_{r \rightarrow 0 }\frac{ r \tau_{n}}{S_{n}(r)}\overline p |\Omega| f''(c_0) c_0^3, \\
\lim_{r \rightarrow 0 } g_{02}=&\lim_{r \rightarrow 0 } \frac{ r \tau_{n}}{S_{n}(r)} e^{2 {\rm i} \nu_r \tau_n}\overline p |\Omega| f''(c_0) c_0^3,\\
\end{aligned}
\end{equation}
and
$$
\begin{aligned}
\lim_{r \rightarrow 0 } g_{21}=&\lim_{r \rightarrow 0 } \frac{2 r \tau_{n}}{S_{n}(r)}  e^{- {\rm i} \nu_r \tau_n} \overline p |\Omega| f''(c_0) c_0^3 (-\frac{{\rm i} g_{11}}{\nu_{r} \tau_{n}} e^{- {\rm i} \nu_r \tau_n}+\frac{{\rm i} \overline {g}_{11}}{\nu_{r} \tau_{n}} e^{ {\rm i} \nu_r \tau_n}+\frac{F}{c_0})\\
&+\lim_{r \rightarrow 0 }\frac{r \tau_{n}}{S_{n}(r)} e^{{\rm i}\nu_r \tau_n}  \overline p |\Omega| f''(c_0) c_0^3 (\frac{{\rm i} g_{20}}{\nu_{r} \tau_{n}} e^{- {\rm i} \nu_r \tau_n}+\frac{{\rm i} \overline{g}_{02}}{3 \nu_{r} \tau_{n}} e^{ {\rm i} \nu_r \tau_n}+E e^{-2 {\rm i} \nu_r \tau_n})\\
&+\lim_{r \rightarrow 0 }\frac{r \tau_{n}}{S_{n}(r)} e^{-{\rm i}\nu_r \tau_n}  \overline p |\Omega| f''(c_0) c_0^3  \\
=&\lim_{r \rightarrow 0 }\left\{\frac{2}{\nu_{r} \tau_{n}}\left[-{\rm i}g_{11}g_{20}+{\rm i}|g_{11}|^2\right]+g_{11}e^{-{\rm i}\nu_r \tau_n}\frac{2F}{c_0}\right\}\\
&+\lim_{r \rightarrow 0 }\left[ \frac{{\rm i}g_{11}g_{20}}{\nu_{r} \tau_{n}}+\frac{{\rm i}|g_{02}|^2}{3\nu_{r} \tau_{n}}+g_{11}e^{-{\rm i} \nu_r \tau_n}\frac{E}{c_0}\right]\\
&+\lim_{r \rightarrow 0 }\left[-g_{11}e^{-{\rm i} \nu_r \tau_n}c_0+g_{11}e^{-{\rm i} \nu_r \tau_n}\frac{c_0}{c_0-2}\right].
\end{aligned}
$$
Therefore,
$$
\begin{aligned}
\lim_{r \rightarrow 0 } \mathcal{R}e g_{21}=&\lim_{r \rightarrow 0 } \mathcal{R}e \left[\frac{-{\rm i}g_{11}g_{20}}{\nu_{r} \tau_{n}} +g_{11}e^{-{\rm i}\nu_r \tau_n}\left(\frac{E}{c_0}+\frac{2F}{c_0}+\frac{c_0}{c_0-2}-c_0\right)\right].
\end{aligned}
$$
This, combined with Eq. \eqref{C1}, implies that
\begin{equation}
\begin{aligned}
\lim_{r \rightarrow 0}\mathcal{R}e C_{1}(0)=&\lim_{r \rightarrow 0}\mathcal{R}e \left[\frac{\rm i}{2 \nu_r \tau_{n}}\left(g_{11} g_{20}-2\left|g_{11}\right|^{2}-\frac{\left|g_{02}\right|^{2}}{3}\right)+\frac{g_{21}}{2}\right]\\
=&\lim_{r \rightarrow 0}\mathcal{R}e \left(\frac{{\rm i} g_{11} g_{20}}{2 \nu_r \tau_{n}}+\frac{g_{21}}{2} \right)\\
=&\lim_{r \rightarrow 0}\mathcal{R}e \left[\frac{{\rm i} g_{11} g_{20}}{2 \nu_r \tau_{n}}+ \frac{-{\rm i}g_{11}g_{20}}{2\nu_{r} \tau_{n}} +\frac{1}{2}g_{11}e^{-{\rm i}\nu_r \tau_n}\left(\frac{E}{c_0}+\frac{2F}{c_0}+\frac{c_0}{c_0-2}-c_0\right) \right]\\
=&\frac{1}{2}\lim_{r \rightarrow 0} \mathcal{R}e\left[ g_{11}e^{-{\rm i}\nu_r \tau_n}\left(\frac{E}{c_0}+\frac{2F}{c_0}+\frac{c_0}{c_0-2}-c_0\right)\right]\\
=&\frac{1}{2}\lim_{r \rightarrow 0}\left[\mathcal{R}e (g_{11}e^{-{\rm i} \nu_r \tau_n})(\mathcal{R}e\frac{E}{c_0}+\frac{2F}{c_0}+\frac{c_0}{c_0-2}-c_0)\right.\\
&\left.- \mathcal{I} m  \left(g_{11}e^{-{\rm i} \nu_r \tau_n}\right)\mathcal{I}m\frac{E}{c_0} \right].
\end{aligned}
\end{equation}
In order to analyze the sign of $\lim_{r \rightarrow 0}\mathcal{R}e C_{1}(0)$, we only need to calculate the signs of $\lim_{r \rightarrow 0}\mathcal{R}e (g_{11}e^{-{\rm i} \nu_r \tau_n})$, $\lim_{r \rightarrow 0}(\mathcal{R}e\frac{E}{c_0}+\frac{2F}{c_0}+\frac{c_0}{c_0-2}-c_0)$, $\lim_{r \rightarrow 0} \mathcal{I} m  \left(g_{11}e^{-{\rm i} \nu_r \tau_n}\right)$ and $\lim_{r \rightarrow 0}\mathcal{I}m\frac{E}{c_0} $, respectively.
From \eqref{Sn(r)}, we have
$$\lim_{r \rightarrow 0}S_n(r)=\lim_{r \rightarrow 0} (1+r \tau_n f'(c_0)e^{-{\rm i} \theta_0}\overline p )c_0^2 |\Omega|,$$
then
\begin{equation}\label{Sr}
\begin{aligned}
\lim_{r \rightarrow 0} \frac{1}{S_n(r)}&=\lim_{r \rightarrow 0}\frac{1+r \tau_n f'(c_0)e^{{\rm i} \theta_0}\overline p }{(1+r \tau_n f'(c_0)e^{-{\rm i} \theta_0}\overline p )(1+r \tau_n f'(c_0)e^{{\rm i} \theta_0}\overline p )c_0^2 |\Omega|}\\
&=\lim_{r \rightarrow 0} \frac{1+r \tau_n f'(c_0)\overline p \cos\theta_0+{\rm i} r \tau_n f'(c_0)\overline p\sin\theta_0}{\left[1+2r \tau_n f'(c_0)\overline p \cos\theta_0+(r \tau_n f'(c_0)\overline p)^2\right]c_0^2 |\Omega|}.
\end{aligned}
\end{equation}
It follows from \eqref{g201102} and \eqref{Sr} that
$$\begin{aligned}
\lim_{r \rightarrow 0} g_{11}&=\lim_{r \rightarrow 0 }\frac{ r \tau_{n}}{S_{n}(r)}\overline p |\Omega| f''(c_0) c_0^3 \\
&=\lim_{r \rightarrow 0} \frac{ r \tau_{n}\overline p f''(c_0) c_0\left[1+r \tau_n f'(c_0)\overline p \cos\theta_0+{\rm i} r \tau_n f'(c_0)\overline p\sin\theta_0\right]}{[1+2r \tau_n f'(c_0)\overline p \cos\theta_0+(r \tau_n f'(c_0)\overline p)^2]}.
\end{aligned}$$
Then, together with Eq. \eqref{A1}, yields
\begin{equation}\label{Reg11e}
\begin{aligned}
\lim_{r \rightarrow 0} \mathcal{R}e \left(g_{11}e^{-{\rm i} \theta_0}\right)=&\lim_{r \rightarrow 0} \left\{\mathcal{R}e g_{11} \cos\theta_0+  \mathcal{I} m  g_{11} \sin\theta_0\right\}\\
=&\lim_{r \rightarrow 0} \frac{ r \tau_{n}\overline p f''(c_0) c_0\left(1+r \tau_n f'(c_0)\overline p \cos\theta_0\right)\cos\theta_0}{[1+2r \tau_n f'(c_0)\overline p \cos\theta_0+(r \tau_n f'(c_0)\overline p)^2]}\\
&+\lim_{r \rightarrow 0} \frac{ r \tau_{n}\overline p f''(c_0) c_0\left(r \tau_n f'(c_0)\overline p\sin\theta_0\right)\sin\theta_0}{[1+2r \tau_n f'(c_0)\overline p \cos\theta_0+(r \tau_n f'(c_0)\overline p)^2]}\\
=&\lim_{r \rightarrow 0} \frac{r \tau_{n}\overline p f''(c_0) c_0 \cos\theta_0+ \left(r \tau_{n}\overline p\right)^2 f'(c_0) f''(c_0) c_0}{[1+2r \tau_n f'(c_0)\overline p \cos\theta_0+(r \tau_n f'(c_0)\overline p)^2]}\\
=&\frac{(\theta_0+2n\pi) \sqrt{c_0^2 - 2{c_0}}\frac{1}{1-c_0}+(\theta_0+2n\pi)^2(1-c_0)}{1+2\frac{(\theta_0+2n\pi)}{\sqrt{c_0^2 - 2{c_0}}}+\left[\frac{(\theta_0+2n\pi)(1-c_0)}{\sqrt{c_0^2 - 2{c_0}}}\right]^2},
\end{aligned}
\end{equation}
and
\begin{equation}\label{Img11e}
\begin{aligned}
\lim_{r \rightarrow 0} \mathcal{I}m \left(g_{11}e^{-{\rm i} \theta_0}\right)=&\lim_{r \rightarrow 0}\left\{ -\mathcal{R}e g_{11} \sin\theta_0+  \mathcal{I} m  g_{11} \cos\theta_0\right\}\\
=&-\lim_{r \rightarrow 0} \frac{ r \tau_{n}\overline p f''(c_0) c_0\left(1+r \tau_n f'(c_0)\overline p \cos\theta_0\right)\sin\theta_0}{[1+2r \tau_n f'(c_0)\overline p \cos\theta_0+(r \tau_n f'(c_0)\overline p)^2]}\\
&+\lim_{r \rightarrow 0} \frac{ r \tau_{n}\overline p f''(c_0) c_0\left(r \tau_n f'(c_0)\overline p\sin\theta_0\right)\cos\theta_0}{[1+2r \tau_n f'(c_0)\overline p \cos\theta_0+(r \tau_n f'(c_0)\overline p)^2]}\\
=&\lim_{r \rightarrow 0} \frac{- r \tau_{n}\overline p f''(c_0) c_0\sin\theta_0}{[1+2r \tau_n f'(c_0)\overline p \cos\theta_0+(r \tau_n f'(c_0)\overline p)^2]}\\
=&\frac{(\theta_0+2n\pi)(c_0 - 2)c_0\frac{1}{1-c_0}}{1+2\frac{(\theta_0+2n\pi)}{\sqrt{c_0^2 - 2{c_0}}}+\left[\frac{(\theta_0+2n\pi)(1-c_0)}{\sqrt{c_0^2 - 2{c_0}}}\right]^2}.
\end{aligned}
\end{equation}
Note that
$$
\begin{aligned}
\lim_{r\rightarrow 0 } \frac{E}{c_0}=& \frac{-e^{-2{\rm i}\theta_0 }(c_0-2)c_0}{e^{-2{\rm i}\theta_0 }(1-c_0)-1-2{\rm i} \sqrt {c_0^2 - 2{c_0}}}\\
=&(-\cos2\theta_0+{\rm i}\sin2\theta_0) (c_0-2)c_0\\
&\times\frac{[(1-c_0)\cos2\theta_0-1]+ {\rm i}[(1-c_0)\sin2\theta_0+2\sqrt{c_0^2 - 2{c_0}}]}{((1-c_0)\cos2\theta_0-1)^2+ ((1-c_0)\sin2\theta_0+2\sqrt {c_0^2 - 2{c_0}})^2}.
\end{aligned}
$$
Therefore, we obtain
\begin{equation}\label{ReEc0}
\begin{aligned}
\lim_{r\rightarrow 0 } \mathcal{R}e \frac{E}{c_0}=&(c_0-2)c_0\\
&\times\frac{-(1-c_0)\cos{^2}2\theta_0+\cos2\theta_0 -(1-c_0)\sin{^2}2\theta_0-2 \sqrt{c_0^2 - 2{c_0}}\sin2\theta_0}{((1-c_0)\cos2\theta_0-1)^2+ ((1-c_0)\sin2\theta_0+2\sqrt {c_0^2 - 2{c_0}})^2}\\
=&\frac{\left[(c_0-1)+\frac{1+3c_0^2-6c_0}{(1-c_0)^2}\right](c_0-2)c_0}{((1-c_0)\cos2\theta_0-1)^2+ ((1-c_0)\sin2\theta_0+2\sqrt {c_0^2 - 2{c_0}})^2},
\end{aligned}
\end{equation}
and
\begin{equation}\label{ImEc0}
\begin{aligned}
\lim_{r\rightarrow 0 } \mathcal{I}m \frac{E}{c_0}=&\frac{\left[-\sin2\theta_0-2\sqrt{c_0^2 - 2{c_0}}\cos2\theta_0\right](c_0-2)c_0}{((1-c_0)\cos2\theta_0-1)^2+ ((1-c_0)\sin2\theta_0+2\sqrt {c_0^2 - 2{c_0}})^2}\\
=&\frac{\frac{2\sqrt {c_0^2 - 2{c_0}}}{(1-c_0)^2}(c_0^2-2c_0)^2}{((1-c_0)\cos2\theta_0-1)^2+ ((1-c_0)\sin2\theta_0+2\sqrt {c_0^2 - 2{c_0}})^2}\\
=&\frac{2(c_0^2-2c_0)^{\frac{5}{2}}}{5c_0^4-14c_0^3+9c_0^2}.
\end{aligned}
\end{equation}
From \eqref{A1}, one also have
\begin{equation}\label{2Fc0}
\begin{aligned}
\lim_{r\rightarrow 0 }\frac{2F}{c_0}=&\frac{-2\overline p f''(c_0)c_0}{\overline p f'(c_0)-\overline \delta}=\frac{-2(c_0-2)c_0}{-c_0}=2(c_0-2)>0.
\end{aligned}
\end{equation}
Then we see from \eqref{ReEc0} and \eqref{2Fc0} that
\begin{equation}\label{ReE/c0-c0-c0/c0-2}
\begin{aligned}
\lim_{r \rightarrow 0}&(\mathcal{R}e\frac{E}{c_0}+\frac{2F}{c_0}+ \frac{c_0}{c_0-2}-c_0)\\
&=\frac{\left[(c_0-1)+\frac{1+3c_0^2-6c_0}{(1-c_0)^2}\right](c_0-2)c_0}{((1-c_0)\cos2\theta_0-1)^2+ ((1-c_0)\sin2\theta_0+2\sqrt {c_0^2 - 2{c_0}})^2}+(c_0-4)+\frac{c_0}{c_0-2}\\
&=\frac{(c_0^3-3c_0)(c_0-2)^2c_0+(c_0-4)(5c_0^4-14c_0^3+9c_0^2)(c_0-2)+c_0(5c_0^4-14c_0^3+9c_0^2)}{(5c_0^4-14c_0^3+9c_0^2)(c_0-2)}\\
&=\frac{(c_0^2-3)(c_0-2)^2c_0^2+(5c_0^4-14c_0^3+9c_0^2)(c_0^2-5c_0+8)}{(5c_0^4-14c_0^3+9c_0^2)(c_0-2)}>0.
\end{aligned}
\end{equation}
It follows from \eqref{Reg11e}, \eqref{Img11e}, \eqref{ImEc0} and \eqref{ReE/c0-c0-c0/c0-2} that
\begin{equation}\label{A2}
\begin{aligned}
&\lim_{r \rightarrow 0}\mathcal{R}e C_{1}(0)=\lim_{r \rightarrow 0} \mathcal{R}e\left[ g_{11}e^{-{\rm i}\nu_r \tau_n}\left(\frac{E}{c_0}+\frac{2F}{c_0}+\frac{c_0}{c_0-2}-c_0\right)\right]\\
=&\lim_{r \rightarrow 0}\left[\mathcal{R}e (g_{11}e^{-{\rm i} \nu_r \tau_n})(\mathcal{R}e\frac{E}{c_0}+\frac{2F}{c_0}+\frac{c_0}{c_0-2}-c_0)- \mathcal{I} m  \left(g_{11}e^{-{\rm i} \nu_r \tau_n}\right)\mathcal{I}m\frac{E}{c_0} \right]\\
=&\frac{(\theta_0+2n\pi) \sqrt{c_0^2 - 2{c_0}}\frac{1}{1-c_0}+(\theta_0+2n\pi)^2(1-c_0)}{1+2\frac{(\theta_0+2n\pi)}{\sqrt{c_0^2 - 2{c_0}}}+\left[\frac{(\theta_0+2n\pi)(1-c_0)}{\sqrt{c_0^2 - 2{c_0}}}\right]^2}\\
&\times \frac{(c_0^2-3)(c_0-2)^2c_0^2+(5c_0^4-14c_0^3+9c_0^2)(c_0^2-5c_0+8)}{(5c_0^4-14c_0^3+9c_0^2)(c_0-2)}\\
&-\frac{(\theta_0+2n\pi)(c_0 - 2)c_0\frac{1}{1-c_0}}{1+2\frac{(\theta_0+2n\pi)}{\sqrt{c_0^2 - 2{c_0}}}+\left[\frac{(\theta_0+2n\pi)(1-c_0)}{\sqrt{c_0^2 - 2{c_0}}}\right]^2} \times \frac{2(c_0^2-2c_0)^{\frac{5}{2}}(c_0-2)}{(5c_0^4-14c_0^3+9c_0^2)(c_0-2)}.
\end{aligned}
\end{equation}
For simplicity, we only calculate the numerator of \eqref{A2}:
\begin{equation}\label{A3}
\begin{aligned}
&\left[(\theta_0+2n\pi) \sqrt{c_0^2 - 2{c_0}}\frac{1}{1-c_0}+(\theta_0+2n\pi)^2(1-c_0)\right]\\
&\times \left[(c_0^2-3)(c_0-2)^2c_0^2+(5c_0^4-14c_0^3+9c_0^2)(c_0^2-5c_0+8)\right]\\
&-(\theta_0+2n\pi)(c_0 - 2)c_0\frac{1}{1-c_0}\times 2(c_0^2-2c_0)^{\frac{5}{2}}(c_0-2)\\
=&\frac{2(c_0^2-2c_0)^{\frac{7}{2}}(c_0-2)(\theta_0+2n\pi)}{c_0-1}\\
&-\left[(c_0^2-3)(c_0-2)^2c_0^2+(5c_0^4-14c_0^3+9c_0^2)(c_0^2-5c_0+8)\right]\\
&\times\left[(c_0-1)(\theta_0+2n\pi)^2+\frac{\sqrt{c_0^2 - 2{c_0}}(\theta_0+2n\pi)}{c_0-1}\right]\\
<&\frac{2(c_0^2-2c_0)^{\frac{7}{2}}(c_0-2)(\theta_0+2n\pi)}{c_0-1}\\
&-\left[(c_0^2-3)(c_0-2)^2c_0^2+(5c_0^4-14c_0^3+9c_0^2)(c_0^2-5c_0+8)\right](c_0-1)(\theta_0+2n\pi)\\
&-\left[(c_0^2-3)(c_0-2)^2c_0^2+(5c_0^4-14c_0^3+9c_0^2)(c_0^2-5c_0+8)\right]\frac{\sqrt{c_0^2 - 2{c_0}}(\theta_0+2n\pi)}{c_0-1}\\
=&(\theta_0+2n\pi)\left\{\frac{2(c_0^2-2c_0)^{\frac{7}{2}}(c_0-2)}{c_0-1}\right.\\
&\left.-\left[(c_0^2-3)(c_0-2)^2c_0^2+(5c_0^4-14c_0^3+9c_0^2)(c_0^2-5c_0+8)\right](c_0-1)\right\}\\
&-\left[(c_0^2-3)(c_0-2)^2c_0^2+(5c_0^4-14c_0^3+9c_0^2)(c_0^2-5c_0+8)\right]\frac{\sqrt{c_0^2 - 2{c_0}}(\theta_0+2n\pi)}{c_0-1}.
\end{aligned}
\end{equation}
Let
$$\begin{aligned}
A=&\frac{2(c_0^2-2c_0)^{\frac{7}{2}}(c_0-2)}{c_0-1}\\
&-\left[(c_0^2-3)(c_0-2)^2c_0^2+(5c_0^4-14c_0^3+9c_0^2)(c_0^2-5c_0+8)\right](c_0-1),\\
B=&-\left[(c_0^2-3)(c_0-2)^2c_0^2+(5c_0^4-14c_0^3+9c_0^2)(c_0^2-5c_0+8)\right].
\end{aligned}$$
Then, when $c_0>2$, we have
\begin{equation}\label{A}
\begin{aligned}
A\le &\left\{2c_0^2(c_0-2)(c_0-1)^3(c_0-2)^2\right.\\
&\left.-c_0^2(c_0-1)^2[(c_0^2-3)(c_0-2)^2+(c_0-1)(2c_0-4)(c_0^2-5c_0+8)]\right\}\frac{1}{c_0-1}\\
=&\frac{2c_0^2(c_0-2)(c_0-1)^2(c_0-1)(c_0-2)^2-c_0^2(c_0-1)^2(c_0-2)[3c_0^3-14c_0^2+23c_0-10]}{c_0-1}\\
=&\frac{c_0^2(c_0-2)(c_0-1)^2[-c_0^3+4c_0^2-7c_0+2]}{c_0-1}\\
=&\frac{c_0^2(c_0-2)(c_0-1)^2[-c_0(c_0-2)^2-3c_0+2]}{c_0-1}\\
<&0,
\end{aligned}
\end{equation}
and
\begin{equation}\label{B}
\begin{aligned}
B=-\left\{(c_0^2-3)(c_0-2)^2c_0^2+(c_0-1)(5c_0-9)c_0^2\left[(c_0-\frac{5}{2})^2+\frac{7}{4}\right]\right\}<0.
\end{aligned}
\end{equation}
Summarizing the \eqref{A2}, \eqref{A3}, \eqref{A} and \eqref{B}, we have $\lim_{r \rightarrow 0}\mathcal{R}e C_{1}(0)<0$.
\end{proof}



\end{document}